\documentclass[12pt]{amsart}
\pdfoutput=1
\usepackage{graphicx}
\usepackage{amssymb}
\usepackage{amsmath}
\usepackage{amsthm}
\usepackage{amscd}
\usepackage{epsfig}

\newtheorem{theorem}{Theorem}[section]

\newtheorem{corollary}[theorem]{Corollary}
\newtheorem{lemma}[theorem]{Lemma}

\newtheorem{proposition}[theorem]{Proposition}

\newtheorem{remark}[theorem]{Remark}
\theoremstyle{definition}
\newtheorem{definition}[theorem]{Definition}
\numberwithin{equation}{section}

\newcommand{\R}{\mathbb R}
\newcommand{\Z}{\mathbb Z}
\newcommand{\N}{\mathbb N}

\begin{document}

\title[Amenable packing topological entropy]{Packing topological entropy for amenable group actions}

\author [Dou Dou, Dongmei Zheng and Xiaomin Zhou]{Dou Dou*, Dongmei Zheng and Xiaomin Zhou}

\address{D. Dou: Department of Mathematics, Nanjing University,
Nanjing, Jiangsu, 210093, P.R. China} \email{doumath@163.com}
\address{D. Zheng: Department of Applied Mathematics, College of Science, Nanjing Tech University, Nanjing, Jiangsu, 211816, P.R. China} \email{dongmzheng@163.com}
\address{X. Zhou: School of Mathematics and Statistics, Huazhong University of Science and Technology,
Wuhan, Hubei, 430074, P.R. China}
\email{zxm12@mail.ustc.edu.cn}
\thanks{* Corresponding author}

\begin{abstract}
Packing topological entropy is a dynamical analogy of the packing dimension, which can be viewed as a counterpart of Bowen topological entropy.
In the present paper, we will give a systematically study to the packing topological entropy for a continuous $G$-action dynamical system $(X,G)$, where $X$ is a compact metric space and $G$ is a countable infinite discrete amenable group.

We first prove a variational principle for amenable packing topological entropy: for any Borel subset $Z$ of $X$, the packing topological entropy of $Z$ equals the supremum of upper local entropy over all Borel probability measures for which the subset $Z$ has full measure.
And then we obtain an entropy inequality concerning amenable packing entropy. Finally we show that the packing topological entropy of the set of generic points for any invariant Borel probability measure $\mu$ coincides with the metric entropy if either $\mu$ is ergodic or the system satisfies a kind of specification property.
\end{abstract}

\maketitle

%%%%%%%%%%%%%%%%%%%%%%%%%%%%%%%%%%%%%%%%%%%%%%%%%%%%%%%%%%%%%%%%%%%%%%%%%%%%%%%%%%%%%%%%%%%%%%%%%%%%%%%%%%%%%%%%%%%%%%%%%%%%%%%%%%%%
%%%%%%%%%%%%%%%%%%%%%%%%%%%%%%%%%%%%%%%%%%%%%%%%%%%%%%%        Introduction      %%%%%%%%%%%%%%%%%%%%%%%%%%%%%%%%%%%%%%%%%%%%%%%%%%%
%%%%%%%%%%%%%%%%%%%%%%%%%%%%%%%%%%%%%%%%%%%%%%%%%%%%%%%%%%%%%%%%%%%%%%%%%%%%%%%%%%%%%%%%%%%%%%%%%%%%%%%%%%%%%%%%%%%%%%%%%%%%%%%%%%%%
\section{Introduction}
In 1973, in his profound and lasting paper \cite{B2}, Bowen introduced a definition of topological entropy of subset inspired by Hausdorff dimension,
which is now known as Bowen topological entropy or dimensional entropy. For dynamical systems over compact Hausdorff spaces,
Bowen showed that Bowen topological entropy on the whole space coincides with the Adler-Konheim-McAndrew topological entropy defined through open covers.

Bowen topological entropy can be viewed as a dynamical analogy of Hausdorff dimension and it exhibits very deep connections
with dimension theory in dynamical system and multifractal analysis ever since its appearance (see, for example, \cite{P}). It is natural to consider the analogous
concepts in dynamical systems for other forms of dimensions. For pointwise dimension (of a measure), its dynamical correspondence is the Brin-Katok's local entropy (\cite{BK}).
For packing dimension, its dynamical correspondence is the packing topological entropy,
which was introduced by Feng and Huang \cite{FH}. Applying the methods in geometric measure theory, they also provided variational principles for Bowen topological entropy and packing topological entropy.
Other works on packing topological entropy can be found in \cite{ZCC}, where the packing topological entropy for certain noncompact subsets was considered.

In this paper, we will focus on packing topological entropy in the framework of countable discrete amenable group actions.

\subsection{Amenable packing entropy and local entropies}\

Let $(X,G)$ be a $G-$action topological dynamical system, where $X$ is a compact metric space with metric $d$ and $G$ is a topological group acting continuously on $X$.
Throughout this paper, we always assume that $G$ is a countable infinite discrete amenable group unless it is specially mentioned.
Recall that a countable discrete group $G$ is {\it amenable} if there is a sequence of non-empty
finite subsets $\{F_n\}$ of $G$ which are asymptotically invariant, i.e.
$$\lim_{n\rightarrow+\infty}\frac{|F_n\vartriangle gF_n|}{|F_n|}=0, \text{ for all } g\in G.$$
Such sequences are called {\it F{\o}lner sequences}. One may refer to \cite{KL,OW} for more details on amenable groups and their actions.
A F{\o}lner sequence $\{F_n\}$ in $G$ is said to be {\it tempered} if there exists a constant $C>0$ which is independent of $n$ such that
\begin{align}\label{tempered}
|\bigcup_{k<n}F_k^{-1}F_n|\le C|F_n|, \text{ for any }n.
\end{align}

Let $F(G)$ denote the collection of finite subsets of $G$. For $F\in F(G)$, let $d_F$ be the metric defined by
\begin{align*}
  d_F(x,y)=\max_{g\in F}d(gx,gy), \text{ for }x,y\in X.
\end{align*}
Let $\varepsilon>0$ and $x$ in $X$, we denote by
\begin{align*}
B_{F}(x,\varepsilon)&=\{y\in X: d_{F}(x,y)<\varepsilon\}
\end{align*}
and
\begin{align*}
  \overline{B}_{F}(x,\varepsilon)&=\{y\in X: d_{F}(x,y)\le \varepsilon\},
\end{align*}
which are respectively the open and closed {\it ($F$-)Bowen balls} with center $x$ and radius $\varepsilon$.
When we want to clarify the underlying metric $d$, we also denote the above balls by $B_{F}(x,\varepsilon,d)$ and $\overline{B}_{F}(x,\varepsilon,d)$.
For a $\Z$-action (or $\N$-action) topological dynamical system $(X,T)$ ($T$ is the homeomorphism (or the continuous onto map) on $X$), let $F=\{0,1,\cdots,n-1\}:=[0,n-1]$. The ($F$-)Bowen balls will be written as $B_n(x,\varepsilon,T)$ or $B_n(x,\varepsilon,d)$.
%\subsection{Amenable packing topological entropy}
\begin{definition}
  Let $\{F_n\}$ be a sequence of finite subsets of $G$ with $|F_n|\rightarrow \infty$ (need not to be F{\o}lner).
For $Z\subseteq X, s\ge0, N\in\N$ and $\varepsilon> 0$, define
$$P(Z,N,\varepsilon,s,\{F_n\}) = \sup\sum_i\exp\bigg(-s|F_{n_i}|\bigg),$$
where the supremum is taken over all finite or countable pairwise disjoint families $\{\overline{B}_{F_{n_i}}(x_i,\varepsilon)\}$ such that
$x_i\in Z, n_i\ge N$ for all $i$. The quantity $P(Z,N,\varepsilon,s,\{F_n\})$
does not increase as $N$ increases, hence the following limit exists:
$$P(Z,\varepsilon,s,\{F_n\})=\lim_{N\rightarrow+\infty}P(Z,N,\varepsilon,s,\{F_n\}).$$
Define
\begin{align*}
  \mathcal{P}(Z,\varepsilon,s,\{F_n\})=\inf \{\sum_{i=1}^{+\infty}P(Z_i,\varepsilon,s,\{F_n\}): \bigcup_{i=1}^{+\infty}Z_i\supset Z\}.
\end{align*}
It is easy to see that if $Z\subseteq \bigcup\limits_{i=1}^{+\infty}Z_i$, then $\mathcal{P}(Z,\varepsilon,s,\{F_n\})\le \sum\limits_{i=1}^{+\infty}\mathcal{P}(Z_i,\varepsilon,s,\{F_n\})$.
There exists a critical value of the parameter $s$, which we will denote by $h_{top}^{P}(Z,\varepsilon, \{F_n\})$, where $\mathcal{P}(Z,\varepsilon,s,\{F_n\})$ jumps from $+\infty$ to $0$, i.e.,
\begin{equation*}
\mathcal{P}(Z,\varepsilon,s,\{F_n\})=\begin{cases}
                             0, \quad\  s > h_{top}^P(Z,\varepsilon, \{F_n\}),\\
                             +\infty, s < h_{top}^P(Z,\varepsilon, \{F_n\}).
                        \end{cases}
\end{equation*}
It is not hard to see that $h_{top}^P(Z,\varepsilon, \{F_n\})$ increases when $\varepsilon$ decreases. We call
\begin{align*}
  h_{top}^P(Z,\{F_n\}):=\lim_{\varepsilon\rightarrow 0} h_{top}^P(Z,\varepsilon, \{F_n\})
\end{align*}
the {\it amenable packing topological entropy} ({\it amenable packing entropy} or {\it packing entropy}, for short) of $Z$ (w.r.t. the F{\o}lner sequence $\{F_n\}$).
\end{definition}

%\begin{remark}
%  In general, amenable packing topological entropy does depend on the F{\o}lner sequence $\{F_n\}$. But in the case $Z=X$ (or $Z$ is an invariant closed subset of $X$),
%  the packing entropy of $Z$ coincides with the usual topological entropy (Proposition \ref{prop-packing-equal}).
%\end{remark}
Let $M(X)$ denote the collection of Borel probability measures on $X$.

\begin{definition}
Let $\{F_n\}$ be a sequence of finite subsets of $G$ with $|F_n|\rightarrow \infty$. For $\mu\in M(X)$ and $Z\in \mathcal{B}(X)$ (the Borel $\sigma$-algebra on $X$), denote by
$$\overline h_{\mu}^{loc}(Z,\{F_n\})=\int_Z\lim_{\varepsilon\rightarrow 0}\limsup_{n\rightarrow +\infty}-\frac{1}{|F_n|}\log \mu(B_{F_n}(x,\varepsilon))d\mu$$
and
$$\underline h_{\mu}^{loc}(Z,\{F_n\})=\int_Z\lim_{\varepsilon\rightarrow 0}\liminf_{n\rightarrow +\infty}-\frac{1}{|F_n|}\log \mu(B_{F_n}(x,\varepsilon))d\mu,$$
which are called the {\it upper local entropy} and the {\it lower local entropy} of $\mu$ over $Z$ (w.r.t $\{F_n\}$), respectively.

For $\Z$-action or $\N$-action case, if the sequence $\{F_n\}$ is chosen to be $F_n=[0,n-1]$ (hence $\{F_n\}$ is naturally a F{\o}lner sequence), the local entropies of a topological dynamical system $(X,T)$ will be denoted by $\underline h_{\mu}^{loc}(Z,T)$ and $\overline h_{\mu}^{loc}(Z,T)$, respectively.
\end{definition}
By the Brin-Katok's entropy formula for amenable group actions (see \cite{ZC2}), if $\mu$ is in addition $G$-invariant and $\{F_n\}$ is a tempered
F{\o}lner sequence which satisfies the growth condition
\begin{equation}\label{eq-1-1}
    \lim\limits_{n\rightarrow+\infty}\frac{|F_n|}{\log n}=+\infty,
\end{equation}
then the values of the upper and lower local entropies over the whole space $X$ coincide with the measure-theoretic entropy of the system $(X,G)$.

We will prove the following variational principle between amenable packing entropy and upper local entropy.

\begin{theorem}\label{th-viriational-pack}\
Let $(X,G)$ be a $G-$action topological dynamical system and $G$ a countable infinite discrete amenable group. Let $\{F_n\}$ be a sequence of finite subsets in $G$ satisfying the growth condition \eqref{eq-1-1}.
Then for any non-empty Borel subset $Z$ of $X$,
  $$h^P_{top}(Z,\{F_n\})=\sup\{\overline h_{\mu}^{loc}(Z,\{F_n\}):\mu\in M(X),\mu(Z)=1\}.$$
\end{theorem}

\subsection{Amenable packing entropy inequalities via factor maps}\

Let $(X,G)$ and $(Y,G)$ be two $G$-action topological dynamical systems.  A continuous map $\pi:(X,G)\to(Y,G)$ is called a homomorphism or a factor map from $(X,G)$ to $(Y,G)$ if it is onto and $\pi\circ g=g\circ\pi$, for all $g\in G$. We also say that $(X,G)$ is an extension of $(Y,G)$ or $(Y,G)$ is a factor of $(X,G)$.

For a subset $Z$ of $X$, we denote by $h_{top}^{UC}(Z,\{F_n\})$ the upper capacity topological entropy of $Z$ (defined in Section 2). We will show in section 2 that the packing entropy can be estimated via the upper capacity topological entropy with parameters (Proposition \ref{packing-uc}). This is a dynamical version of the fact that the packing dimension can be defined via the upper Minkowski dimension. Applying Proposition \ref{packing-uc}, we can prove the following packing entropy inequalities for factor maps.

\begin{theorem}\label{thm2}\
Let $G$ be a countable infinite discrete amenable group and $\pi:(X,G)\to(Y,G)$ be a factor map between two $G$-action topological dynamical systems. Let $\{F_n\}$  be any tempered F{\o}lner sequence in $G$ satisfying the growth condition \eqref{eq-1-1}.
Then for any Borel subset $E$ of $X$,
\begin{equation}\label{eq-1-2}
    h^P_{top}(\pi (E),\{F_n\})\leq  h^P_{top}( E,\{F_n\})\leq  h^P_{top}(\pi (E),\{F_n\})+\sup_{y\in Y}h_{top}^{UC}(\pi^{-1}(y),\{F_n\}).
\end{equation}
\end{theorem}

We remark here that for $\Z$-actions, the inequalities was proved in \cite{ZCHZ}. But for amenable group actions, except employing Bowen's idea in \cite{B1} and the quasi-tiling techniques developed by Ornstein and Weiss \cite{OW}, we need a crucial covering lemma for amenable groups built by Lindenstrauss while proving pointwise theorems for amenable groups in \cite{L}.

\subsection{Amenable packing entropy for certain subsets}\

Let $M(X,G)$ and $E(X,G)$ be the collection of $G$-invariant and ergodic $G$-invariant Borel probability measures on $X$ respectively.
Since $G$ is amenable, $M(X,G)$ and $E(X,G)$ are both non-empty. For $\mu\in M(X,G)$, let $h_{\mu}(X,G)$ denote the measure-theoretic entropy
of $(X,G)$ w.r.t. $\mu$.

For $\mu\in M(X,G)$ and a F{\o}lner sequence $\{F_n\}$ in $G$, let $G_{\mu,\{F_n\}}$ be the set of generic points for $\mu$ (w.r.t. $\{F_n\}$),
which is defined by
\begin{align*}
  G_{\mu,\{F_n\}}=\{x\in X: \lim_{n\rightarrow+\infty}\frac{1}{|F_n|}\sum_{g\in F_n}f(gx)=\int_Xf d\mu, \text{ for any }f\in C(X)\}.
\end{align*}

For simplicity, we write $G_{\mu,\{F_n\}}$ by $G_{\mu}$ when there is no ambiguity on $\{F_n\}$. But we should note that for different F{\o}lner sequence $\{F_n\}$,
the corresponding $G_{\mu}$ may not coincide. When $\mu$ is ergodic and the F{\o}lner sequence $\{F_n\}$ is tempered,
%by the pointwise ergodic theorem (Lindenstrauss \cite{L}),
$G_{\mu}$ has full measure for $\mu$ (see Remark \ref{remark-5-1}). But when $\mu$ is not ergodic, the set $G_{\mu}$ could be empty.

For the case $G=\Z$, Bowen \cite{B2} proved that the Bowen topological entropy of $G_{\mu}$ equals
the measure-theoretic entropy of $\mu$ if $\mu$ is ergodic. Pfister and Sullivan \cite{PS} extended Bowen's result to the system with the so called
$g$-almost product property for invariant Borel probability measure $\mu$. And the results for amenable group action version were proved by Zheng and Chen \cite{ZC2} and Zhang \cite{Zh} respectively. The $g$-almost product property (see \cite{PS}) is an extension of the specification property for $\Z$-systems
and was generalized to amenable systems in \cite{Zh} (called almost specification property there). It was shown in \cite{Zh} that weak specification
implies almost specification for amenable systems. We leave the detailed definitions of almost specification and weak specification in section 5.

We will prove for packing entropy the following theorem.
\begin{theorem}\label{th-generic}
Let $(X,G)$ be a $G-$action topological dynamical system with $G$ a countable infinite discrete amenable group and let $\mu\in M(X,G)$ and $\{F_n\}$ be a F{\o}lner sequence in $G$ satisfying the growth condition \eqref{eq-1-1}.
If either
$\mu$ is ergodic and $\{F_n\}$ is tempered
or $(X,G)$ satisfies the almost specification property,
then
\begin{align}\label{packing-generic}
  h^P_{top}(G_{\mu}, \{F_n\})=h_{\mu}(X,G).
\end{align}
\end{theorem}

In geometric measure theory, a set is said to be {\it regular} (or ``dimension-regular'') if it has equal Hausdorff and packing dimensions (\cite {T}). As a counterpart in dynamical systems, we have the following definition.
\begin{definition}\label{def-regular}
  A subset is said to be {\it regular} in the sense of dimensional entropy (or regular, for simplicity) if it has equal Bowen entropy and packing entropy.
\end{definition}

An affirmative task in the study of packing entropy is to compute the dimensional entropies (including packing entropy and its dual, Bowen entropy) for various subsets and to make clear which subsets are regular. From Theorem \ref{th-generic} and results in \cite{ZC2,Zh} for Bowen entropy, under the conditions in Theorem \ref{th-generic}, the set $G_{\mu}$ is regular since both of its dimensional entropies equal the measure-theoretic entropy for $\mu$.

{\it Organization of the paper}. In section 2, we will give some properties of amenable packing entropy including connections with Bowen entropy and upper
capacity entropy.
In section 3, we will devote to the proof of Theorem \ref{th-viriational-pack}, the variational principle between amenable packing entropy and upper local entropy. This extends
Feng and Huang's result in \cite{FH} from $\Z$-actions to amenable group actions. In section 4, we are going to prove the packing entropy inequalities for factor maps (Theorem \ref{thm2}). Finally in section 5, we will give the proof of Theorem \ref{th-generic} and provide some examples to discuss the regularity for certain subsets. These examples include subsets of symbolic dynamical system and fibers of the $(T,T^{-1})$ transformation. Some detailed computations via these examples are included in Appendix A.

%%%%%%%%%%%%%%%%%%%%%%%%%%%%%%%%%%%%%%%%%%%%%%%%%%%%%%%%%%%%%%%%%%%%%%%%%%%%%%%%%%%%%%%%%%%%%%%%%%%%%%%%%%%%%%%%%%%%%%%%%%%%%%%%%%%%
%%%%%%%%%%%%%%%%%%%%%%%%%%%%%%%%%%%%%%%%%%%%%%%%%%    Properties of amenable Packing entropy      %%%%%%%%%%%%%%%%%%%%%%%%%%%%%%%%%%%%%%%%%%%%%%
%%%%%%%%%%%%%%%%%%%%%%%%%%%%%%%%%%%%%%%%%%%%%%%%%%%%%%%%%%%%%%%%%%%%%%%%%%%%%%%%%%%%%%%%%%%%%%%%%%%%%%%%%%%%%%%%%%%%%%%%%%%%%%%%%%%%

\section{Properties of amenable packing entropy}
Due to the definition of packing entropy in section 1, it is not hard to prove that the packing entropy has the following properties.
\begin{proposition}\label{prop-basic}
Let $\{F_n\}$ be a sequence of finite subsets in $G$ with $|F_n|\rightarrow \infty$, $Z,Z'$ and $Z_i$ $(i=1,2,\cdots)$ be subsets of $X$.
  \begin{enumerate}
    \item If $Z\subseteq Z'$, then
    $$h_{top}^P(Z,\{F_n\})\le h_{top}^P(Z',\{F_n\}).$$
    \item If $Z\subseteq \bigcup\limits_{i=1}^{+\infty}Z_i$, then for any $\varepsilon>0$,
    $$h_{top}^P(Z,\varepsilon,\{F_n\})\le \sup_{i\ge 1}h_{top}^P(Z_i,\varepsilon,\{F_n\}).$$
    Hence
    $$h_{top}^P(Z,\{F_n\})\le \sup_{i\ge 1}h_{top}^P(Z_i,\{F_n\}).$$
    \item If $\{F_{n_k}\}$ is a subsequence of $\{F_n\}$, then
    $$h_{top}^P(Z,\{F_{n_k}\})\le h_{top}^P(Z,\{F_n\}).$$
  \end{enumerate}
\end{proposition}

%\subsection{Amenable Bowen topological entropy} \

In the following we recall the definition of amenable Bowen topological entropy which was introduced in \cite{ZC}.

Let $\{F_n\}$ be a sequence of finite subsets in $G$ with $|F_n|\rightarrow \infty$. For $Z\subseteq X, s\ge0, N\in\N$ and $\varepsilon> 0$, define
$$\mathcal{M}(Z,N,\varepsilon,s,\{F_n\}) = \inf\sum_i\exp(-s|F_{n_i}|),$$
where the infimum is taken over all finite or countable families $\{B_{F_{n_i}}(x_i,\varepsilon)\}$ such that
$x_i\in X, n_i\ge N$ and $\bigcup\limits_i B_{F_{n_i}}(x_i,\varepsilon)\supseteq Z$. The quantity $\mathcal{M}(Z,N,\varepsilon,s,\{F_n\})$
does not decrease as $N$ increases and $\varepsilon$ decreases, hence the following limits exist:
$$\mathcal{M}(Z,\varepsilon,s,\{F_n\})=\lim_{N\rightarrow+\infty}\mathcal{M}(Z,N,\varepsilon,s,\{F_n\})$$
and
$$\mathcal{M}(Z,s,\{F_n\})=\lim_{\varepsilon\rightarrow0}\mathcal{M}(Z,\varepsilon,s,\{F_n\}).$$
The {\it Bowen topological entropy} $h^B_{top}(Z,\{F_n\})$ is then defined as the critical value
of the parameter $s$, where $\mathcal{M}(Z,s,\{F_n\})$ jumps from $+\infty$ to $0$, i.e.,
\begin{equation*}
\mathcal{M}(Z,s,\{F_n\})=\begin{cases}
                             0, \quad \ s > h^B_{top}(Z,\{F_n\}),\\
                             +\infty, s < h^B_{top}(Z,\{F_n\}).
                        \end{cases}
\end{equation*}

Next we will compare packing topological entropy with Bowen topological entropy.

\begin{proposition}\label{prop-packing-Bowen}
Let $\{F_n\}$ be a sequence of finite subsets in $G$ with $|F_n|\rightarrow \infty$. For any $Z\subseteq X$,
$$h_{top}^B(Z,\{F_n\})\le h_{top}^P(Z,\{F_n\}).$$
\end{proposition}
\begin{proof}
The proof for the case $h_{top}^B(Z,\{F_n\})=0$ is obvious. Now assume $h^B_{top}(Z,\{F_n\})>0$ and let $0< s <h^B_{top}(Z,\{F_n\})$.

Let $\{Z_i\}_{i=1}^{+\infty}$ be any covering of $Z$. For each $i$, for any $n\in \N$ and $\varepsilon>0$,
let $\{\overline B_{F_n}(x_{i,j}, \varepsilon)\}_{j=1}^{N_i}$ be a disjoint subfamily of $\{\overline B_{F_n}(x, \varepsilon)\}_{x\in Z_i}$
with maximal cardinality $N_i$. Then
$$\bigcup_{i=1}^{N_i}B_{F_n}(x_{i,j},3\varepsilon)\supseteq Z_i.$$
So
$$\mathcal M(Z_i,n,3\varepsilon,s,\{F_n\})\le N_ie^{-|F_n|s}\le P(Z_i,n,\varepsilon,s,\{F_n\}),$$
and hence
$$\mathcal M(Z_i,3\varepsilon,s,\{F_n\})\le P(Z_i,\varepsilon,s,\{F_n\}).$$
Thus
\begin{align*}
  \mathcal M(Z,3\varepsilon,s,\{F_n\})\le \sum_{i=1}^{+\infty}M(Z_i,3\varepsilon,s,\{F_n\})
  \le \sum_{i=1}^{+\infty} P(Z_i,\varepsilon,s,\{F_n\}),
\end{align*}
which deduces that
$$\mathcal M(Z,3\varepsilon,s,\{F_n\})\le \mathcal P(Z,\varepsilon,s,\{F_n\}).$$

Since $s<h^B_{top}(Z,\{F_n\})$, we have $\mathcal M(Z,s,\{F_n\})=+\infty$.
So
$$1\le \mathcal M(Z,3\varepsilon,s,\{F_n\})\le\mathcal P(Z,\varepsilon,s,\{F_n\})$$
whenever $\varepsilon$ is sufficiently small.
This implies that
$$h_{top}^P(Z,\varepsilon,\{F_n\})\ge s.$$
Letting $\varepsilon\rightarrow 0$, we have $h_{top}^P(Z,\{F_n\})\ge s$.
Hence
$$h_{top}^B(Z,\{F_n\})\le h_{top}^P(Z,\{F_n\}).$$
\end{proof}

\begin{corollary}\label{coro-packing}
Let $\mu\in M(X,G)$, $Z\subseteq X$ with $\mu(Z)=1$ and $\{F_n\}$ be a F{\o}lner sequence in $G$, then
$$h_{\mu}(X,G)\le h_{top}^P(Z,\{F_n\}).$$
\end{corollary}
\begin{proof}
Let $\{F_{n_k}\}$ be a subsequence of $\{F_n\}$ which is tempered and satisfies the growth condition \eqref{eq-1-1}.
By Theorem 1.2 of \cite{ZC2}, $h_{\mu}(X,G)\le h_{top}^B(Z,\{F_{n_k}\})$. Together with Proposition \ref{prop-packing-Bowen} and (3) of Proposition \ref{prop-basic},
we have
$$h_{\mu}(X,G)\le h_{top}^P(Z,\{F_{n_k}\})\le h_{top}^P(Z,\{F_{n}\}).$$
\end{proof}

%\subsection{Upper capacity topological entropy}\

Let $\varepsilon> 0$, $Z\subseteq X$ and $F\in F(G)$. A subset $E\subseteq Z$ is said to be an {\it $(F,\varepsilon)$-separated set} of $Z$, if
for any two distinct points $x, y\in E$, $d_F(x, y)>\varepsilon$. Let $s_F(Z, \varepsilon)$ denote the largest cardinality of $(F, \varepsilon)$-separated sets for $Z$. A subset $E\subseteq X$ is said to be an {\it $(F,\varepsilon)$-spanning set} of $Z$, if for any $x\in Z$, there exists $y\in E$
with $d_{F}(x,y)\le\varepsilon$. Let $r_F(Z, \varepsilon)$ (sometimes we use $r_F(Z, \varepsilon,d)$ to indicate the accompanied metric $d$) denote the smallest cardinality of $(F, \varepsilon)$-spanning sets for $Z$. For a sequence of finite subsets $\{F_n\}$ in $G$ with $|F_n|\rightarrow \infty$,
the {\it upper capacity topological entropy} of $Z$ is defined as
\begin{align*}
  h_{top}^{UC}(Z,\{F_n\})&=\lim_{\varepsilon\rightarrow 0}\limsup_{n\rightarrow +\infty}\frac{1}{|F_n|}\log s_{F_n}(Z, \varepsilon)=\lim_{\varepsilon\rightarrow 0}\limsup_{n\rightarrow +\infty}\frac{1}{|F_n|}\log r_{F_n}(Z, \varepsilon).
\end{align*}
The second equality comes from the following simple fact:
\begin{align}\label{sep-span}
  r_{F_n}(Z, 2\varepsilon)\le s_{F_n}(Z, 2\varepsilon) \le r_{F_n}(Z, \varepsilon).
\end{align}
For convention we denote by
$$h_{top}^{UC}(Z,\varepsilon,\{F_n\})=\limsup_{n\rightarrow +\infty}\frac{1}{|F_n|}\log s_{F_n}(Z, \varepsilon),$$
and then $$h_{top}^{UC}(Z,\{F_n\})=\lim_{\varepsilon\rightarrow 0}h_{top}^{UC}(Z,\varepsilon,\{F_n\}).$$
We note here that for the case $Z=X$ and $\{F_n\}$ is a F{\o}lner sequence, the quantity $h_{top}^{UC}(X,\{F_n\})$ coincides with $h_{top}(X,G)$, the topological entropy of $(X,G)$.
\begin{proposition}\label{prop-packing-uc}
Let $\{F_n\}$ be a sequence of finite subsets in $G$ satisfying the growth condition \eqref{eq-1-1}, then for any subset $Z$ of $X$ and any $\varepsilon>0$,
$$h_{top}^P(Z,\varepsilon,\{F_n\})\le h_{top}^{UC}(Z,\varepsilon,\{F_n\}).$$
Hence
$$h_{top}^P(Z,\{F_n\})\le h_{top}^{UC}(Z,\{F_n\}).$$
\end{proposition}
\begin{proof}
Let $\varepsilon>0$ be fixed. The proposition is obvious for the case $h_{top}^P(Z,\varepsilon,\{F_n\})=0$.
Assume $h_{top}^P(Z,\varepsilon, \{F_n\})>0$ and let $0<t<s<h_{top}^P(Z,\varepsilon,\{F_n\})$. Then
$$P(Z,\varepsilon,s,\{F_n\})\ge \mathcal P(Z,\varepsilon,s,\{F_n\})=+\infty.$$
Thus for any $N$, there exists a countable pairwise disjoint family $\{\overline{B}_{F_{n_i}}(x_i, \varepsilon)\}$ with $x_i\in Z$ and $n_i\ge N$ for all $i$ such that
$1< \sum\limits_ie^{-|F_{n_i}|s}$. For each $k$, let $m_k$ be the number of $i$ with $n_i=k$. Then we have
$$\sum_ie^{-|F_{n_i}|s}=\sum_{k\ge N}m_ke^{-|F_k|s}.$$
Since $\{F_n\}$ satisfies the growth condition $\lim\limits_{n\rightarrow+\infty}\frac{|F_n|}{\log n}=+\infty$, $\sum\limits_{k\ge 1}e^{|F_k|(t-s)}$ converges.
Let $M=\sum\limits_{k\ge 1}e^{|F_k|(t-s)}$. There must be some $k\ge N$ such that $m_k>\frac{1}{M}e^{|F_k|t}$, otherwise the above sum is at most
$$\sum_{k\ge 1}\frac{1}{M}e^{|F_k|t}e^{-|F_k|s}=1.$$
So $s_{F_k}(Z, \varepsilon)\ge m_k > \frac{1}{M}e^{|F_k|t}$ and hence
$$h_{top}^{UC}(Z,\varepsilon,\{F_n\})=\limsup_{k\rightarrow +\infty}\frac{1}{|F_k|}\log s_{F_k}(Z, \varepsilon)\ge t,$$
which deduces that $h_{top}^P(Z,\varepsilon,\{F_n\})\le h_{top}^{UC}(Z,\varepsilon,\{F_n\})$.
\end{proof}
As a corollary, we have
\begin{corollary}\label{prop-packing-equal}
If $\{F_n\}$ is a F{\o}lner sequence that satisfies the growth condition \eqref{eq-1-1}, then
$$h_{top}^P(X,\{F_n\})=h_{top}(X,G).$$
\end{corollary}
\begin{proof}
By Corollary \ref{coro-packing}, for any $\mu\in M(X,G)$, $h_{\mu}(X,G)\le h_{top}^P(X,\{F_n\})$. Applying the variational principle for amenable topological entropy (cf. \cite{OP,ST}),
we have $h_{top}^P(X,\{F_n\})\ge h_{top}(X,G)$. By Proposition \ref{prop-packing-uc}, we have $h_{top}^P(X,\{F_n\})\le h_{top}(X,G)$.
\end{proof}
\begin{remark}\label{remark-2.6}
  \begin{enumerate}
    \item By Theorem 1.1 of \cite{DZ} (see also \cite{ZC}), if $\{F_n\}$ is a tempered F{\o}lner sequence and satisfies the growth condition \eqref{eq-1-1}, then
    $$h_{top}^P(X,\{F_n\})= h_{top}(X,G)=h_{top}^B(X,\{F_n\}).$$
    \item If we replace $X$ by a $G$-invariant compact subset, the above equality also holds and hence any $G$-invariant compact subset is regular (when the F{\o}lner sequence $\{F_n\}$ is tempered and satisfies the growth condition \eqref{eq-1-1}).
  \end{enumerate}
\end{remark}
%\subsection{An equivalent definition for amenable packing entropy}\
At the end of this section, we will give further relations between packing entropy and upper capacity topological entropy.
\begin{proposition}\label{packing-uc}
  Let $\varepsilon>0$, $Z$ be a subset of $X$ and let $\{F_n\}$ be a sequence of finite subsets in $G$ satisfying the growth condition \eqref{eq-1-1}.
  \begin{enumerate}
    \item
    \begin{align*}
    h^P_{top}(Z,\varepsilon,\{F_n\})\le \inf\{\sup_{i\ge 1} h^{UC}_{top}(Z_i,\varepsilon,\{F_n\}): Z=\cup_{i=1}^{\infty}Z_i\}.
  \end{align*}
  Hence
    \begin{align*}
    h^P_{top}(Z,\{F_n\})\le\inf\{\sup_{i\ge 1} h^{UC}_{top}(Z_i,\{F_n\}): Z=\cup_{i=1}^{\infty}Z_i\}.
  \end{align*}
    \item For any $\delta>0$, there exists a cover $\cup_{i=1}^{\infty}Z_i=Z$ (which depends on both $\varepsilon$ and $\delta$) such that
  \begin{align*}
    h^P_{top}(Z,\varepsilon,\{F_n\})+\delta\ge \sup_{i\ge 1} h^{UC}_{top}(Z_i,3\varepsilon,\{F_n\}).
  \end{align*}
  \end{enumerate}

\end{proposition}
\begin{proof}
For any $Z=\cup_{i=1}^{\infty}Z_i$, by Propositions \ref{prop-basic} and \ref{prop-packing-uc},
\begin{align*}
   h_{top}^P(Z,\varepsilon,\{F_n\})\le \sup_{i\ge 1}h_{top}^P(Z_i,\varepsilon,\{F_n\})\le \sup_{i\ge 1} h^{UC}_{top}(Z_i,\varepsilon,\{F_n\}).
\end{align*}
Hence $$h^P_{top}(Z,\varepsilon,\{F_n\})\le \inf\{\sup_{i\ge 1} h^{UC}_{top}(Z_i,\varepsilon,\{F_n\}): Z=\cup_{i=1}^{\infty}Z_i\}.$$

For the opposite direction, we may assume that $h^P_{top}(Z,\varepsilon,\{F_n\})<\infty$. Let $\delta>0$ be fixed and set $s=h^P_{top}(Z,\varepsilon,\{F_n\})+\delta$.
By the definition of the amenable packing entropy, we have that $\mathcal{P}(Z,\varepsilon,s,\{F_n\})=0$.
Then there exists a cover $\cup_{i=1}^{\infty}Z_i\supseteq Z$ such that
$$\sum_{i\ge 1}P(Z_i,\varepsilon,s,\{F_n\})<1.$$
For each $Z_i$, when $N$ is large enough, we have $P(Z_i,N,\varepsilon,s,\{F_n\})<1$. Let $E$ be any $(F_N,3\varepsilon)$-separated subset of $Z_i$.
Noting that the closed Bowen balls $\overline B_{F_N}(x_i,\varepsilon)$ ($x_i\in E\subset Z_i$) are pairwise disjoint, we have
\begin{align*}
  |E|e^{-s|F_N|}=\sum_{x_i\in E}e^{-s|F_N|}\le P(Z_i,N,\varepsilon,s,\{F_n\})<1.
\end{align*}
Hence $s_{F_N}(Z_i,3\varepsilon)<e^{s|F_N|}$, which leads to $h^{UC}_{top}(Z_i,3\varepsilon,\{F_n\})\le s$. Thus we have
\begin{align*}
  h_{top}^P(Z,\varepsilon,\{F_n\})+\delta \ge \sup_{i\ge 1} h^{UC}_{top}(Z_i,3\varepsilon,\{F_n\}).
\end{align*}
\end{proof}
\begin{remark}
  Proposition \ref{prop-packing-uc} is motivated from the equivalent definition of packing dimension through Minkowski dimension in geometric measure theory, which was due to Tricot \cite{T} (see also \cite{M} for reference). It is unclear for us whether it holds that
\begin{align*}
    h^P_{top}(Z,\{F_n\})= \inf\{\sup_{i\ge 1} h^{UC}_{top}(Z_i,\{F_n\}): Z=\cup_{i=1}^{\infty}Z_i\}.
  \end{align*}
\end{remark}

%%%%%%%%%%%%%%%%%%%%%%%%%%%%%%%%%%%%%%%%%%%%%%%%%%%%%%%%%%%%%%%%%%%%%%%%%%%%%%%%%%%%%%%%%%%%%%%%%%%%%%%%%%%%%%%%%%%%%%%%%%%%%%%%%%%%
%%%%%%%%%%%%%%%%%%%%%%%%%%%%%%%%%%%%%%%%%%%%%%%%%%%%%%%%%  A variation principle for packing entropy%%%%%%%%%%%%%%%%%%%%%%%%%%%%%%%%%%%%%%%
%%%%%%%%%%%%%%%%%%%%%%%%%%%%%%%%%%%%%%%%%%%%%%%%%%%%%%%%%%%%%%%%%%%%%%%%%%%%%%%%%%%%%%%%%%%%%%%%%%%%%%%%%%%%%%%%%%%%%%%%%%%%%%%%%%%%
\section{A variational principle for amenable packing entropy}
In this section we will give the proof of our Theorem \ref{th-viriational-pack}, the variational principle for amenable packing topological entropy.
We assume in this section $\{F_n\}$ is a sequence of finite subsets in $G$ and satisfies the growth condition \eqref{eq-1-1}.

\subsection{Lower bound}

\begin{proposition}\label{prop-lower}
 Let $Z\subseteq X$ be a Borel set, then
 $$h^P_{top}(Z, \{F_n\})\ge\sup\{\overline h_{\mu}^{loc}(Z,\{F_n\}):\mu\in M(X),\mu(Z)=1\}.$$
\end{proposition}

For the proof, we need the following classical $5r$-Lemma in geometric measure theory (cf. \cite{M}, Theorem 2.1).

\begin{lemma}[$5r$-Lemma]\label{lemma-5r}
Let $(X, d)$ be a compact metric space and $\mathcal B =\{B(x_i, r_i)\}_{i\in I}$ be a family of closed (or open) balls in $X$.
Then there exists a finite or countable subfamily $\mathcal B'=\{B(x_i, r_i)\}_{i\in I'}$ of pairwise disjoint balls in $\mathcal B$
such that
$$\bigcup_{B\in\mathcal B}B\subseteq \bigcup_{i\in I'}B(x_i,5r_i).$$
\end{lemma}

We also need the following lemma, which comes from the definition of the packing entropy directly.

\begin{lemma}\label{lemma-3-3}
  Let $E\subset X$ and $s>0$. Then for any $0<\varepsilon_1<\varepsilon_2$,
  $$P(\overline E, \varepsilon_2, s, \{F_n\})\le P(E, \varepsilon_1, s, \{F_n\}).$$
\end{lemma}

\begin{proof}[Proof of Proposition \ref{prop-lower}]
Let $\mu\in M(X)$ with $\mu(Z)=1$ and assume $\overline h_{\mu}^{loc}(Z,\{F_n\})>0$.

Let $0<s<\overline h_{\mu}^{loc}(Z,\{F_n\})$. Then there exist $\varepsilon,\delta>0$ and a Borel set $A\subset Z$ with $\mu(A)>0$ such that for every $x\in A$, it holds that
$$\overline h_{\mu}(x,\varepsilon,\{F_n\})>s+\delta,$$
where $\overline h_{\mu}(x,\varepsilon,\{F_n\}):=\limsup\limits_{n\rightarrow+\infty}-\frac{1}{|F_n|}\log \mu(B_{F_n}(x,\varepsilon))$.

Let $E\subset A$ be any Borel set with $\mu(E)>0$. Define
$$E_n=\{x\in E:\mu(B_{F_n}(x,\varepsilon))<e^{-|F_n|(s+\delta)}\}, n\in\N.$$
Then $\bigcup_{n=N}^{+\infty}E_n=E$ for any $N\in\N$, and hence $\mu(\bigcup_{n=N}^{+\infty}E_n)=\mu(E)$. Fix $N\in \N$, there exists $n\ge N$ such that
$$\mu(E_n)\ge \frac{1}{n(n+1)}\mu(E).$$

Fix such $n$ and consider the family of Bowen balls $\{B_{F_n}(x,\frac{\varepsilon}{5}):x\in E_n\}$. By the $5r$-Lemma \ref{lemma-5r} (we use metric $d_{F_n}$ instead of $d$),
there exists a finite pairwise disjoint family $\{B_{F_n}(x_i,\frac{\varepsilon}{5})\}$ with $x_i\in E_n$ such that
$$\cup_{i}B_{F_n}(x_i,\varepsilon)\supset \cup_{x\in E_n}B_{F_n}(x,\frac{\varepsilon}{5})\supset E_n.$$
Thus
\begin{align*}
  P(E,N,\frac{\varepsilon}{5},s,\{F_n\})&\ge P(E_n,N,\frac{\varepsilon}{5},s,\{F_n\})\ge \sum_{i}e^{-|F_n|s}=e^{|F_n|\delta}\sum_{i}e^{-|F_{n}|(s+\delta)}\\
   &\ge e^{|F_n|\delta}\sum_{i}\mu(B_{F_n}(x_i,\varepsilon))\ge e^{|F_n|\delta}\mu(E_n)\ge \frac{e^{|F_n|\delta}}{n(n+1)}\mu(E).
\end{align*}
By the growth condition \eqref{eq-1-1} of the sequence $\{F_n\}$, we have
$$\frac{e^{|F_n|\delta}}{n(n+1)}\rightarrow +\infty, \text{ as }n\rightarrow +\infty.$$
Letting $N\rightarrow+\infty$, we obtain that
$$P(E,\frac{\varepsilon}{5},s,\{F_n\})=+\infty.$$
Note that this equality holds for every Borel set $E\subset A$ with $\mu(E)>0$.

Let $\{A_i\}_{i=1}^\infty$ be any covering of $A$. Then by Lemma \ref{lemma-3-3},
$$\sum_i P(A_i,\frac{\varepsilon}{10},s,\{F_n\})\ge \sum_i P(\overline A_i,\frac{\varepsilon}{5},s,\{F_n\}).$$
Since there must exist some $A_i$ such that $\overline A_i\cap A$ (which is a Borel set now)
contains a Borel subset $E\subset \overline A_i\cap A$ with $\mu(E)>0$, we have
$$\sum_i P(A_i,\frac{\varepsilon}{10},s,\{F_n\})\ge P(E,\frac{\varepsilon}{5},s,\{F_n\})=+\infty.$$
Thus
$$\mathcal P (Z,\frac{\varepsilon}{10},s,\{F_n\})\ge \mathcal P (A,\frac{\varepsilon}{10},s,\{F_n\})=+\infty,$$
which deduces that
$$h_{top}^P(Z,\{F_n\})\ge h_{top}^P(Z,\frac{\varepsilon}{10},\{F_n\})\ge s.$$

Since $s$ is chosen arbitrarily in $(0,\overline h_{\mu}(Z,\{F_n\}))$, we finally show that
$$h_{top}^P(Z,\{F_n\})\ge \overline h_{\mu}^{loc}(Z,\{F_n\}).$$
This finishes the proof of Proposition \ref{prop-lower}.
\end{proof}

\subsection{Upper bound}

The following Proposition is the upper bound part of the variational principle. In fact it is valid for any analytic set $Z$.
Recall that a set in a metric space is said to be {\it analytic} if it is a continuous image of the set $\mathcal N$ of infinite sequences of natural
numbers. In a Polish space, the collection of analytic subsets contains Borel sets and is
closed under countable unions and intersections (c.f. Chapter 11 of \cite{J}).

\begin{proposition}\label{prop-upper}
  Let $Z\subseteq X$ be an analytic set with $h^P_{top}(Z, \{F_n\})>0$. For any $0<s<h^P_{top}(Z, \{F_n\})$,
  there exists a compact set $K\subseteq Z$ and $\mu\in M(K)$ such that $\overline h_{\mu}^{loc}(K,\{F_n\})\ge s$.
\end{proposition}

The following lemma is needed.

\begin{lemma}\label{lemma-packing}
Let $Z\subseteq X$ and $s,\varepsilon>0.$ Assume that $\mathcal P (Z,\varepsilon,s,\{F_n\})=+\infty$. Then for any given
finite interval $(a,b)\subset \R$ with $a\ge 0$ and any $N\in\N$, there exists a finite disjoint
collection $\{\overline B_{F_{n_i}}(x_i, \varepsilon)\}$ such that $x_i\in Z, n_i\ge N$ and
$\sum_i\exp(-|F_{n_i}|s)\in (a,b)$.
\end{lemma}
\begin{proof}
  See Lemma 4.1 of \cite{FH}.
\end{proof}

\begin{proof}[Proof of Proposition \ref{prop-upper}]
  Since $Z$ is analytic, there exists a continuous surjective map $\phi: \mathcal {N}\rightarrow Z$. Let
$\Gamma_{n_1,n_2,...,n_p}=\{(m_1,m_2, \ldots)\in\mathcal {N}: m_1\le n_1, m_2\le n_2, \ldots, m_p\le n_p\}$
and let $Z_{n_1,\ldots,n_p}=\phi(\Gamma_{n_1,\ldots,n_p})$.

For $0<s<h^P_{top}(Z, \{F_n\})$, take $\varepsilon>0$ small enough to make $0<s<h^P_{top}(Z,\varepsilon,\{F_n\})$ and take $t\in (s, h^P_{top}(Z,\varepsilon,\{F_n\}))$.
Followed by Feng-Huang's steps (which are inspired by the work of Joyce and Preiss \cite{JP} on packing measures),
we will construct inductively the following data.
\begin{enumerate}
  \item[(D-1)] A sequence of finite sets $(K_i)_{i=1}^{+\infty}$ with $K_i\subset Z$.
  \item[(D-2)] A sequence of finite measures $(\mu_i)_{i=1}^{+\infty}$ with each $\mu_i$ supported on $K_i$.
  \item[(D-3)] A sequence of integers $(n_i)_{i=1}^{+\infty}$ and a sequence of positive numbers $(\gamma_i)_{i=1}^{+\infty}$.
  \item[(D-4)] A sequence of integer-valued functions $(m_i: K_i\rightarrow \N)_{i=1}^{+\infty}$.
\end{enumerate}

Moreover, the sequences $(K_i),(\mu_i), (n_i),(\gamma_i)$ and $(m_i(\cdot))$ will be constructed to satisfy the following conditions.
  \begin{enumerate}
    \item [(C-1)] For each $i$, the family $\mathcal V_i:=\{B(x,\gamma_i)\}_{x\in K_i}$ is disjoint. Each element in
$\mathcal V_{i+1}$ is a subset of $B(x,\gamma_i/2)$ for some $x\in K_i$.
    \item [(C-2)] For each $i$, $K_i\subset Z_{n_1,\ldots,n_i}$ and $\mu_i=\sum_{y\in K_i}e^{-|F_{m_i(y)}|s}\delta_y$ with $1<\mu_1(K_1)<2$.
    \item [(C-3)] For each $x\in K_i$ and $z\in B(x,\gamma_i)$,
    \begin{align}\label{eq-balli}
      \overline B_{F_{m_i(x)}}(z,\varepsilon)\cap \bigcup_{y\in K_i\setminus \{x\}}B(y,\gamma_i)=\emptyset
    \end{align}
    and
    \begin{align}\label{ineq-mui}
      \mu_i(B(x,\gamma_i))=e^{-|F_{m_i(x)}|s}\le \sum_{y\in E_{i+1}(x)}e^{-|F_{m_{i+1}(y)}|s}<(1+2^{-(i+1)})\mu_i(B(x,\gamma_i)),
    \end{align}
where $E_{i+1}(x)=B(x,\gamma_i)\cap K_{i+1}$.
  \end{enumerate}
We will give the construction later.

Suppose the sequences $(K_i),(\mu_i), (n_i),(\gamma_i)$ and $(m_i(\cdot))$ have been constructed.
By \eqref{ineq-mui}, for $V_i\in \mathcal V_i$,
\begin{align*}
  \mu_i(V_i)\le \mu_{i+1}(V_i)=\sum_{V\in\mathcal V_{i+1},V\subset V_i}\mu_{i+1}(V)\le (1+2^{-(i+1)})\mu_i(V_i).
\end{align*}
Using the above inequalities repeatedly, we have for any $j>i$ and any $V_i\in \mathcal V_i$,
\begin{align}\label{ineq-mui2}
  \mu_i(V_i)\le \mu_j(V_i)\le \prod_{n=i+1}^j(1+2^{-n})\mu_i(V_i)\le C\mu_i(V_i),
\end{align}
where $C:=\prod_{n=1}^{+\infty}(1+2^{-n})<+\infty$.

Let $\tilde \mu$ be a limit point of $(\mu_i)$ in the weak$^*$ topology. Let
$$K=\bigcap_{n=1}^{+\infty}\overline{\bigcup_{i\ge n}K_i}.$$
Then $\tilde \mu$ is supported on $K$. Furthermore,
$$K\subset\bigcap_{p=1}^{+\infty}\overline{Z_{n_1,\ldots,n_p}}.$$

By the continuity of $\phi$, applying Cantor's diagonal argument, we can show that
$\bigcap_{p=1}^{+\infty}Z_{n_1,\ldots,n_p}=\bigcap_{p=1}^{+\infty}\overline{Z_{n_1,\ldots,n_p}}$.
Hence $K$ is a compact subset of $Z$.

By \eqref{ineq-mui2}, for any $x\in K_i$,
$$e^{-|F_{m_i(x)}|s}=\mu_i(B(x,\gamma_i))\le \tilde \mu(B(x,\gamma_i))\le C\mu_i(B(x,\gamma_i))=Ce^{-|F_{m_i(x)}|s}.$$
In particular,
$$1\le \sum_{x\in K_1}\mu_1(B(x,\gamma_1))\le \tilde \mu(K)\le \sum_{x\in K_1}C\mu_1(B(x,\gamma_1))\le 2C.$$
Note that $K\subset \bigcup_{x\in K_i}B(x,\gamma_i/2)$. By \eqref{eq-balli}, the first part of (C-3), for each $x\in K_i$ and $z\in \overline B(x,\gamma_i)$,
$$\tilde \mu(\overline B_{F_{m_i(x)}}(z,\varepsilon))\le\tilde \mu(\overline B(x,\gamma_i/2))\le Ce^{-|F_{m_i(x)}|s}.$$
For each $z\in K$ and $i\in\N$, $z\in \overline B(x,\gamma_i/2)$ for some $x\in K_i$. Hence
$$\tilde \mu(B_{F_{m_i(x)}}(z,\varepsilon))\le Ce^{-|F_{m_i(x)}|s}.$$

Define $\mu=\tilde \mu/\tilde \mu(K)$. Then $\mu\in M(K)$. For each $z\in K$, there exists a sequence $k_i\uparrow+\infty$ such that
$\mu(B_{F_{k_i}}(z,\varepsilon))\le Ce^{-|F_{k_i}|s}/\tilde \mu(K)$. Hence $\overline h_{\mu}^{loc}(K,\{F_n\})\ge s$.

Now the only thing left is to give the inductive construction of the data $(K_i),(\mu_i)$, $(n_i),(\gamma_i)$ and $(m_i(\cdot))$.
The inductive steps are as follows.

{\bf Step $1$.} Construct the data $K_1,\mu_1, n_1,\gamma_1$ and $m_1(\cdot)$.

Since $t\in (s, h^P_{top}(Z,\varepsilon,\{F_n\}))$, we have that $\mathcal P (Z,\varepsilon,t,\{F_n\})=+\infty$.
Let
$$H=\cup\{U\subset X : U \text{ is open}, \mathcal P (Z\cap U,\varepsilon,t,\{F_n\})=0\}.$$
Then by the separability of $X$, $H$ is a countable union of the open sets $U$'s. Hence $\mathcal P (Z\cap H,\varepsilon,t,\{F_n\})=0$.
Let $Z'=Z\setminus H=Z\cap(X\setminus H)$. If $\mathcal P (Z'\cap U,\varepsilon,t,\{F_n\})=0$ for some open set $U$, then
\begin{align*}
  \mathcal P (Z\cap U,\varepsilon,t,\{F_n\})\le \mathcal P (Z'\cap U,\varepsilon,t,\{F_n\})+\mathcal P (Z\cap H,\varepsilon,t,\{F_n\})=0.
\end{align*}
So $U\subset H$ and then $Z' \cap U=\emptyset$. Hence for any open set $U\subset X$, either $Z' \cap U =\emptyset$ or $\mathcal P (Z' \cap U,\varepsilon,t,\{F_n\})>0$.

Because $\mathcal P (Z,\varepsilon,t,\{F_n\})\le \mathcal P (Z',\varepsilon,t,\{F_n\})+\mathcal P (Z\cap H,\varepsilon,t,\{F_n\})=\mathcal P (Z',\varepsilon,t,\{F_n\})$,
we have $\mathcal P (Z',\varepsilon,t,\{F_n\})= \mathcal P (Z,\varepsilon,t,\{F_n\})=+\infty$. Then $\mathcal P (Z',\varepsilon,s,\{F_n\})=+\infty$.

By Lemma \ref{lemma-packing}, we can find a finite set $K_1\subset Z' $, an integer valued function $m_1(x)$ on $K_1$ such that the collection
$\{\overline B_{F_{m_1(x)}}(x,\varepsilon)\}_{x\in K_1}$ is disjoint and
$$\sum_{x\in K_1}e^{-|F_{m_1(x)}|s}\in (1,2).$$
Define $\mu_1=\sum_{x\in K_1} e^{-|F_{m_1(x)}|s}\delta_x$, where $\delta_x$ denotes the Dirac measure at $x$.
Take $\gamma_1>0$ sufficiently small such that for any function $z:K_1\rightarrow X$ with $\max_{x\in K_1}d(x, z(x))\le \gamma_1$, we have
for each $x\in K_1$,
\begin{align}\label{eq-ball}
  \bigg( \overline B(z(x),\gamma_1)\cup \overline B_{F_{m_1(x)}}(z(x),\varepsilon)\bigg )\cap
  \bigg(\cup_{y\in K_1\setminus \{x\}}\overline B(z(y),\gamma_1)\cup \overline B_{F_{m_1(y)}}(z(y),\varepsilon)\bigg)=\emptyset.
\end{align}

Since $K_1\subset Z'$, we have $\mathcal P (Z\cap B(x,\gamma_1/4),\varepsilon,t,\{F_n\})\ge \mathcal P (Z'\cap B(x,\gamma_1/4),\varepsilon,t,\{F_n\})>0$
for each $x\in K_1$. Therefore we can pick a sufficiently large $n_1\in \N$ so that $Z_{n_1}\supset K_1$ and
$\mathcal P (Z_{n_1}\cap B(x,\gamma_1/4),\varepsilon,t,\{F_n\})>0$ for each $x\in K_1$.

{\bf Step $2$.} Construct the data $K_2,\mu_2, n_2,\gamma_2$ and $m_2(\cdot)$.

By \eqref{eq-ball}, the family of balls $\{B(x,\gamma_1)\}_{x\in K_1}$ are pairwise disjoint. For each $x\in K_1$,
since $\mathcal P (Z_{n_1}\cap B(x,\gamma_1/4),\varepsilon,t,\{F_n\})>0$, similar as Step 1, we can construct a finite set
$E_2(x)\subset Z_{n_1}\cap B(x,\gamma_1/4)$ and an integer-valued function $m_2:E_2(x)\rightarrow \N\cap [\max \{m_1(y):y\in K_1\},+\infty)$ such that
\begin{enumerate}
  \item [(2-a)] for each open set $U$ with $U\cap E_2(x)\neq\emptyset$, $\mathcal P (Z_{n_1}\cap U,\varepsilon,t,\{F_n\})>0$;
  \item [(2-b)] the elements in the family $\{\overline B_{F_{m_2(y)}}(y,\varepsilon)\}_{y\in E_2(x)}$ are pairwise disjoint and
  \begin{align*}
    \mu_1(\{x\})<\sum_{y\in E_2(x)}e^{-|F_{m_2(y)}|s}<(1+2^{-2})\mu_1(\{x\}).
  \end{align*}
\end{enumerate}

To see this, we fix $x\in K_1$. Denote $V=Z_{n_1}\cap B(x, \gamma_1/4)$. Let
$$H_x:=\cup\{U\subset X: U \text{ is open and } \mathcal P (V\cap U,\varepsilon,t,\{F_n\})=0\}.$$
Set $V'=V\setminus H_x$. Then as in Step 1, we can show that
$$\mathcal P (V',\varepsilon,t,\{F_n\})=\mathcal P (V,\varepsilon,t,\{F_n\})>0.$$
Moreover, for any open set $U\subset X$, either $V'\cap U =\emptyset$ or $\mathcal P (V'\cap U,\varepsilon,t,\{F_n\})>0$.
Since $s<t$, it holds that $\mathcal P (V',\varepsilon,s,\{F_n\})=+\infty$. By Lemma \ref{lemma-packing},
we can find a finite set $E_2(x)\subset V'$ and a map $m_2: E_2(x)\rightarrow \N\cap [\max \{m_1(y):y\in K_1\},+\infty)$ such that (2-b) holds.
Notice that if an open set $U$ satisfies $U\cap E_2(x)\neq\emptyset$, then $U\cap V'\neq\emptyset$. So
$\mathcal P (Z_{n_1}\cap U,\varepsilon,t,\{F_n\})\ge \mathcal P (V'\cap U,\varepsilon,t,\{F_n\})>0$. Hence (2-a) holds.

Since the family $\{B(x,\gamma_1)\}_{x\in K_1}$ is disjoint, so is the family $\{E_2(x)\}_{x\in K_1}$.
Define
$$K_2=\cup_{x\in K_1}E_2(x)\text{ and }\mu_2=\sum_{y\in K_2}e^{-|F_{m_2(y)}|s}\delta_y.$$

By \eqref{eq-ball} and (2-b), the elements in $\{\overline B_{F_{m_2(y)}}(y,\varepsilon)\}_{y\in K_2}$ are pairwise disjoint.
Hence we can take $0<\gamma_2<\gamma_1/4$ such that for any function $z:K_2\rightarrow X$ satisfying that $\max_{x\in K_{2}}d(x,z(x))< \gamma_2$,
we have for $x\in K_2$,
\begin{align}\label{eq-ball2}
  \bigg( \overline B(z(x),\gamma_2)\cup \overline B_{F_{m_2(x)}}(z(x),\varepsilon)\bigg )\cap
  \bigg(\cup_{y\in K_2\setminus \{x\}}\overline B(z(y),\gamma_2)\cup \overline B_{F_{m_2(y)}}(z(y),\varepsilon)\bigg)=\emptyset.
\end{align}
Choose a sufficiently large $n_2\in \N$ so that $Z_{n_1,n_2}\supset K_2$ and
$$\mathcal P (Z_{n_1,n_2}\cap B(x,\gamma_2/4),\varepsilon,t,\{F_n\})>0$$
for each $x\in K_2$.

{\bf Step $3$.} Next we suppose that the data $K_i,\mu_i, n_i,\gamma_i$ and $m_i(\cdot)$ ($i=1,\ldots,p$)
have been constructed. We will construct the data $K_{p+1},\mu_{p+1}, n_{p+1},\gamma_{p+1}$ and $m_{p+1}(\cdot)$.

Assume that we have constructed $K_i,\mu_i, n_i,\gamma_i$ and $m_i(\cdot)$ for $i=1,\ldots,p$.
And assume that for any function $z:K_p\rightarrow X$ with $d(x,z(x))<\gamma_p$ for all
$x\in K_p$, we have
\begin{align}\label{eq-ballp}
  \bigg( \overline B(z(x),\gamma_p)\cup \overline B_{F_{m_p(x)}}(z(x),\varepsilon)\bigg )\cap
  \bigg(\cup_{y\in K_p\setminus \{x\}}\overline B(z(y),\gamma_p)\cup \overline B_{F_{m_p(y)}}(z(y),\varepsilon)\bigg)=\emptyset
\end{align}
for each $x\in K_p$; and $Z_{n_1,\ldots,n_p}\supset K_p$ and
$\mathcal P (Z_{n_1,\ldots,n_p}\cap B(x,\gamma_p/4),\varepsilon,t,\{F_n\})>0$ for each $x\in K_p$.

Note that the family of balls $\{\overline B(x,\gamma_p)\}_{x\in K_p}$ are pairwise disjoint. For each $x\in K_p$,
since $\mathcal P (Z_{n_1,\ldots,n_p}\cap B(x,\gamma_p/4),\varepsilon,t,\{F_n\})>0$, similar as Step 2, we can construct a finite set
$E_{p+1}(x)\subset Z_{n_1,\ldots,n_p}\cap B(x,\gamma_p/4)$ and an integer-valued function $m_{p+1}:E_{p+1}(x)\rightarrow \N\cap [\max \{m_p(y):y\in K_p\},+\infty)$ such that
\begin{enumerate}
  \item [(3-a)] for each open set $U$ with $U\cap E_{p+1}(x)\neq\emptyset$, $\mathcal P (Z_{n_1,\ldots,n_p}\cap U,\varepsilon,t,\{F_n\})>0$;
  \item [(3-b)] the elements in the family $\{\overline B_{F_{m_{p+1}(y)}}(y,\varepsilon)\}_{y\in E_{p+1}(x)}$ are pairwise disjoint and
  \begin{align*}
    \mu_p(\{x\})<\sum_{y\in E_{p+1}(x)}e^{-|F_{m_{p+1}(y)}|s}<(1+2^{-(p+1)})\mu_p(\{x\}).
  \end{align*}
\end{enumerate}

By \eqref{eq-ballp} and (3-b), the family $\{\overline B_{F_{m_{p+1}(y)}}(y,\varepsilon)\}_{y\in K_{p+1}}$ is disjoint.
Hence we can take $0<\gamma_{p+1}< \gamma_p/4$ such that for any function $z:K_{p+1}\rightarrow X$ with $\max_{x\in K_{p+1}}d(x,z(x))<\gamma_{p+1}$,
we have for each $x\in K_{p+1}$,
\begin{align}\label{eq-ballp+1}
  \bigg( \overline B(z(x),\gamma_{p+1})\cup &\overline B_{F_{m_{p+1}(x)}}(z(x),\varepsilon)\bigg )\nonumber \\
  &\bigcap
  \bigg(\cup_{y\in K_{p+1}\setminus \{x\}}\overline B(z(y),\gamma_{p+1})\cup \overline B_{F_{m_{p+1}(y)}}(z(y),\varepsilon)\bigg)=\emptyset.
\end{align}
Choose a sufficiently large $n_{p+1}\in \N$ so that $Z_{n_1,\ldots,n_{p+1}}\supset K_{p+1}$ and
$\mathcal P (Z_{n_1,\ldots,n_{p+1}}\cap B(x,\gamma_{p+1}/4),\varepsilon,t,\{F_n\})>0$ for each $x\in K_{p+1}$.

Then we finish the required construction and complete the proof of the Proposition.
\end{proof}
%%%%%%%%%%%%%%%%%%%%%%%%%%%%%%%%%%%%%%%%%%%%%%%%%%%%%%%%%%%%%%%%%%%%%%%%%%%%%%%%%%%%%%%%%%%%%%%%%%
%%%%%%%%%%%%%%%%%%%%%%%%%%%%%%%%%%%%%%%%%%%%%%%%%%%%%%%%%%%%%%%%%%%%%%%%%%%%%%%%%%%%%%%%%%%%%%%%%%

\section{Packing entropy inequalities for factor maps}\ \

In this section, we will prove Theorem \ref{thm2}. The proof is a combination of Bowen's method in Theorem 17 of \cite{B1} and quasi-tiling techniques for amenable groups.

\subsection{Preliminaries for amenable groups}\ \

Let $G$ be a countable infinite discrete amenable group.
Let $A$ and $K$ be two nonempty finite subsets of $G$. Recall that $B(A,K)$, the {\it $K$-boundary} of $A$, is defined by
$$B(A,K)=\{g\in G: Kg\cap A\neq\emptyset \text { and } Kg\cap(G\setminus A)\neq\emptyset\}.$$
For $\delta>0$, the set $A$ is said to be {\it $(K, \delta)$-invariant} if
\begin{align*}
  \frac{|B(A,K)|}{|A|}<\delta.
\end{align*}
We say a sequence of non-empty finite subsets $\{F_n\}$ of $G$ becomes more and more invariant if for any $\delta>0$ and any non-empty finite subset $K$ of $G$, $F_n$ is $(K,\delta)$-invariant for sufficiently large $n$.
An equivalent condition for the sequence $\{F_n\}$ to be a F{\o}lner sequence is that $\{F_n\}$ becomes more and more
invariant (see \cite{OW}).

Let $\tilde{\mathcal{F}}$ be a collection of finite subsets of $G$. It is said to be $\delta$-disjoint if for every $A\in \tilde{\mathcal{F}}$
there exists an $A'\subset A$ such that $|A'|\ge (1-\delta)|A|$ and such that $A'\cap B'=\emptyset$ for every $A\neq B\in \tilde{\mathcal{F}}$.

The following is a covering lemma for amenable groups by Lindenstrauss.

\begin{lemma}[Lindenstrauss's covering lemma, Corollary 2.7 of \cite{L}]\label{lemma-4-1}
  For any $\delta\in (0,1/100), C>0$ and finite $D\subset G$, let $M\in\N$ be sufficiently large (depending only on $\delta,C$ and $D$).
  Let $F_{i,j}$ be an array of finite subsets of $G$ where $i=1,\ldots,M$ and $j=1,\ldots,N_i$, with the following two requirements:
  \item [{\bf Requirement 1.}] For every $i$, $\bar{F}_{i,*}=\{F_{i,j}\}_{j=1}^{N_i}$ satisfies
  $$|\bigcup_{k'<k}F_{i,k'}^{-1}F_{i,k}|\le C|F_{i,k}|,\quad \text{ for }k=2,\ldots,N_i.$$
  \item [{\bf Requirement 2.}] Denote $F_{i,*}=\cup\bar{F}_{i,*}$. The finite set sequences $F_{i,*}$ satisfy that for every $1<i\le M$ and every $1\le k \le N_i$,
  $$|\bigcup_{i'<i}DF_{i',*}^{-1}F_{i,k}|\le (1+\delta)|F_{i,k}|.$$

  Assume that $A_{i,j}$ is another array of finite subsets of $G$ with $F_{i,j}A_{i,j}\subset F$ for some finite subset $F$ of $G$.
  Let $A_{i,*}=\cup_jA_{i,j}$ and $$\alpha=\frac{\min_{1\le i\le M}|DA_{i,*}|}{|F|}.$$

Then the collection of subsets of $F$,
  $$\tilde{\mathcal{F}}=\{F_{i,j}a: 1\le i\le M, 1\le j\le N_i \text{ and }a\in A_{i,j}\}$$
  has a subcollection $\mathcal{F}$ that is $10\delta^{1/4}$-disjoint such that
  $$|\cup \mathcal{F}|\ge (\alpha-\delta^{1/4})|F|.$$
\end{lemma}

\subsection{Proof of Theorem~\ref{thm2}}\ \

Let $G$, $\{F_n\}$ and $\pi:(X,G)\to (Y,G)$ be as in Theorem \ref{thm2} and $E\subset X$ be a Borel set. Let $d$ and $\rho$ be the compatible metrics on $X$ and $Y$, respectively.

For any $\varepsilon>0$, there exists $\delta>0$ such that for any $x_1,x_2\in X$ with $d(x_1,x_2)\le \delta$, one has $\rho(\pi(x_1),\pi(x_2))\le \varepsilon$.
Now let $\{y_i\}_{i=1}^{k}\subset \pi(E)$ be any $(F_n,\varepsilon)$-separated set of $\pi(E)$ and choose for each $i$ a point $x_i\in \pi^{-1}(y_i)\cap E$, then $\{x_i\}_{i=1}^k$ forms an $(F_n,\delta)$-separated set of $E$. Hence
$$s_{F_n}(\pi(E), \varepsilon)\le s_{F_n}(E, \delta)$$
and
\begin{align}\label{UC-factor}
  h^{UC}_{top}( \pi (E), \varepsilon,\{F_n\})\leq  h^{UC}_{top}(E, \delta, \{F_n\}).
\end{align}
%It is not hard to verify that
%$$ h^{UC}_{top}( \pi (E),\{F_n\})\leq  h^{UC}_{top}(E,\{F_n\}).$$
By Proposition \ref{packing-uc} (2), for any $\eta>0$, there exists a cover $\cup_{i=1}^{\infty}E_i=E$ such that
$$h^P_{top}(E,\delta/3,\{F_n\})+\eta \ge \sup_{i\ge 1} h^{UC}_{top}(E_i,\delta,\{F_n\}).$$
Then we have
\begin{align*}
  h^P_{top}(\pi (E),\varepsilon,\{F_n\})&\le\sup_{i\ge 1} h^{P}_{top}(\pi(E_i),\varepsilon,\{F_n\})\quad \text{(by Proposition \ref{prop-basic} (2))}\\
                            &\le\sup_{i\ge 1} h^{UC}_{top}(\pi(E_i),\varepsilon,\{F_n\})\quad \text{(by Proposition \ref{prop-packing-uc})}\\
                            &\le \sup_{i\ge 1} h^{UC}_{top}(E_i,\delta,\{F_n\})\quad \text{(noting that \eqref{UC-factor} also holds for each $E_i$)}\\
                            &\le h^P_{top}(E,\delta/3,\{F_n\})+\eta,
\end{align*}
which implies that
$$h^P_{top}(\pi (E),\{F_n\})\leq  h^P_{top}( E,\{F_n\}).$$
\if
%\subsubsection{Lower bound of $h^P_{top}( E,\{F_n\})$}\ \

Let $\nu\in M(Y)$ with $\nu(\pi(E))=1$, then $\nu\circ\pi\in M(X)$ and $\nu\circ\pi(E)=1$. For any $\varepsilon>0,$ there exists $\delta>0$ such that $\rho(\pi(x_1),\pi(x_2))<\varepsilon$, whenever $d(x_1,x_2)<\delta$.
Then $\pi(B(x,\delta,d))\subset B(\pi(x),\varepsilon,\rho)$ and hence $\pi(B_{F_n}(x,\delta,d))\subset B_{F_n}(\pi(x),\varepsilon,\rho).$
It follows that
\begin{align*}
\overline{h}_{\nu\circ\pi}^{loc}(\{F_n\})
&= \int_X\lim_{\delta\to 0}\limsup_{n\rightarrow+\infty}\frac{-\log \nu\circ\pi(B_{F_n}(x,\delta,d))}{|F_n|}d\nu\circ\pi\\
&\geq \int_X\lim_{\varepsilon\to 0} \limsup_{n\rightarrow+\infty}\frac{-\log \nu(B_{F_n}(\pi(x),\varepsilon,\rho))}{|F_n|}d\nu\circ\pi\\
&\geq \int_Y\lim_{\varepsilon\to 0} \limsup_{n\rightarrow+\infty}\frac{-\log \nu(B_{F_n}(y,\varepsilon,\rho))}{|F_n|}d\nu\\
&=\overline{h}_{\nu}^{loc}(\{F_n\}).
\end{align*}
Thus by Theorem~\ref{th-viriational-pack},
$$h^P_{top}(\pi (E),\{F_n\})\leq  h^P_{top}( E,\{F_n\}).$$

%\subsubsection{Upper bound of $h^P_{top}( E,\{F_n\})$}\ \

\fi

To get the upper bound, we need to prove the following inequality for amenable upper capacity topological entropy, which extends a result by
Bowen (\cite{B1}, Theorem 17) to amenable group actions.
%{\bf We firstly assume that $\{F_n\}$ is tempered.}
\begin{theorem}\label{thm-ineq-uc}\
Let $G$ be a countable infinite discrete amenable group and $\pi:(X,G)\to(Y,G)$ be a factor map between two $G$-action topological dynamical systems. Let $\{F_n\}$  be any tempered F{\o}lner sequence in $G$ satisfying the growth condition \eqref{eq-1-1}.
Then for any subset $E$ of $X$,
\begin{equation}\label{eq-4-1}
    h^{UC}_{top}( E,\{F_n\})\leq  h^{UC}_{top}(\pi (E),\{F_n\})+\sup_{y\in Y}h_{top}^{UC}(\pi^{-1}(y),\{F_n\}).
\end{equation}
\end{theorem}
\begin{proof}
If $\sup_{y\in Y}h_{top}^{UC}(\pi^{-1}(y),\{F_n\})=\infty$ then there is nothing to prove. So we assume that
$$a:=\sup_{y\in Y}h_{top}^{UC}(\pi^{-1}(y),\{F_n\})<\infty.$$

To verify \eqref{eq-4-1}, we need some preparations along the following three steps.

{\bf Step 1.} Construct $F_{i,j}$, the array of subsets of $G$.

Fix $\tau>0$. For any $\varepsilon>0$, let $0<\delta<\min\{\varepsilon,1/100\}$ be small enough.
Let $C>0$ be the constant in the tempered condition \eqref{tempered} for the F{\o}lner sequence $\{F_n\}$ and let $D=\{e_G\}\subset G$, where $e_G$ is the identity element of $G$.
Let $M>0$ be large enough as in Lemma \ref{lemma-4-1} corresponding to $\delta, C$ and $D$.

For each $y\in Y$, choose $m(y)\in \N$ such that for any $n\ge m(y)$,
\begin{align}\label{condition-m(y)}
  \frac{1}{|F_{n}|}\log r_{F_{n}}(\pi^{-1}(y),\varepsilon,d)\leq h_{top}^{UC}(\pi^{-1}(y),\{F_n\})+\tau \leq a+\tau .
\end{align}
Here recall that $r_{F_{n}}(\pi^{-1}(y),\varepsilon,d)$ denotes the smallest cardinality of $(F_{n}, \varepsilon)$-spanning sets for $\pi^{-1}(y)$.
Let $E_y$ be an $(F_{m(y)},\varepsilon)$-spanning set of $\pi^{-1}(y)$ with the smallest cardinality $|E_y|=r_{F_{m(y)}}(\pi^{-1}(y),\varepsilon,d).$ Denote
$$U_y=\{p\in X: \text{ there exists }q\in E_y \text{ such that }d_{F_{m(y)}}(p,q)<2\varepsilon\},$$
which is an open neighborhood of $\pi^{-1}(y)$. Since $\bigcap_{\gamma>0}\pi^{-1}(\overline{B(y,\gamma,\rho)})=\pi^{-1}(y)$, we have
$(X\setminus U_y)\bigcap \big (\bigcap_{\gamma>0}\pi^{-1}(\overline{B(y,\gamma,\rho)})\big )=\emptyset$. Hence
by the finite intersection property of compact sets, there is a $W_y:=B(y,\gamma_y,\rho)$ for some $\gamma_y>0$ such that $U_y\supset \pi^{-1}(W_y).$
Since $Y$ is compact, there exist $y_{1,1},\cdots, y_{1,r_1}$ such that $W_{y_{1,1}},\cdots,W_{y_{1,r_1}}$ cover $Y$.
List the F{\o}lner sets in the collection $\{F_{m(y_{1,k})}: 1\le k\le r_1\}$ by
$$F_{n_{1,1}}, F_{n_{1,2}},\cdots, F_{n_{1,N_1}}, \text{ where }n_{1,1}<n_{1,2}<\cdots<n_{1,N_1}.$$
Note that $N_1=\#\{F_{m(y_{1,k})}: 1\le k\le r_1\}\le r_1$.

For each $y\in Y$, choose $m(y)>n_{1,N_1}$ such that \eqref{condition-m(y)} holds for any $n\ge m(y)$. Repeat the above process, we can obtain
$y_{2,1},\cdots, y_{2,r_2}\in Y$ such that $W_{y_{2,1}},\cdots,W_{y_{2,r_2}}$ cover $Y$.
We then list the F{\o}lner sets in the collection $\{F_{m(y_{2,k})}: 1\le k\le r_2\}$ by
$$F_{n_{2,1}}, F_{n_{2,2}},\cdots, F_{n_{2,N_2}}, \text{ where }n_{2,1}<n_{2,2}<\cdots<n_{2,N_2}.$$
Note that $N_2=\#\{F_{m(y_{2,k})}: 1\le k\le r_2\}\le r_2$.

Repeating the above steps inductively, we can obtain for each $1\le i\le M$,
\begin{enumerate}
  \item a collection of points $y_{i,1},\cdots, y_{i,r_i}\in Y$ such that $W_{y_{i,1}},\cdots,W_{y_{i,r_i}}$ cover $Y$;
  \item a collection of F{\o}lner sets $\{F_{n_{i,1}}, F_{n_{2,2}},\cdots, F_{n_{i,N_i}}\}$($=\{F_{m(y_{i,k})}: 1\le k\le r_i\}$)
  with $n_{i,1}<n_{i,2}<\cdots<n_{i,N_i}$ and $N_i\le r_i$.
\end{enumerate}
From the above construction, for each $1\le i\le M-1$, $n_{i,N_i}<n_{i+1,1}$. Moreover, $n_{i+1,1}$ can be chosen sufficiently large compared with
$n_{i,N_i}$ such that for every $1<i\le M$ and every $1\le k \le N_i$,
  \begin{align}\label{condition-2 of lemma}
    |\bigcup_{i'<i}F_{n_{i'},*}^{-1}F_{n_i,k}|\le (1+\delta)|F_{n_i,k}|.
  \end{align}

For simplification, we denote $F_{i,j}=F_{n_{i,j}}$ for each $1\le i\le M$ and $1\le j\le N_i$, which is the array of $G$ we required.

{\bf Step 2.} Produce quasi-tilings from $F_{i,j}$.

Let $\eta_1$ be a common Lebesgue number of the family of open covers
$\{W_{y_{i,1}},\cdots,W_{y_{i,r_i}}\}$ with respect to the metric $\rho$. Denote $\eta=\frac{\eta_1}{2}$.

Let $N$ be large enough such that for every $n>N$, $F_n$ is
$(F_{i,*}\cup\{e_G\},\delta)$-invariant for all $1\le i\le M$.

For each $y\in Y$ and $n>N$, let
\begin{align*}
  A_{i,j}=&\{a\in F_n: F_{i,j}a\subset F_n \text{ and there exists }1\le k\le r_i \text{ such that } F_{m(y_{i,k})}=F_{i,j} \\
  &\qquad \qquad \text{ and }\overline{B}(ay,\eta,\rho)\subseteq W_{y_{i,k}} \}.
\end{align*}
We note here that $A_{i,j}$ depends on $y$.

For any $g\in F_n\setminus B(F_n,F_{i,*}\cup\{e_G\})$, we have $F_{i,*}g\subset F_n$. Since $\eta_1$($=2\eta$) is a Lebesgue number of
$\{W_{y_{i,1}},\cdots,W_{y_{i,r_i}}\}$, $\overline{B}(gy,\eta,\rho)$ is contained in some $W_{y_{i,k}}$ and then $g\in A_{i,*}$. Hence
$$A_{i,*}\supseteq F_n\setminus B(F_n,F_{i,*}\cup\{e_G\})$$
for each $1\le i \le M$, and
$$\alpha=\frac{\min_{1\le i\le M}|DA_{i,*}|}{|F_n|}=\frac{\min_{1\le i\le M}|A_{i,*}|}{|F_n|}>1-\delta.$$

Now we are able to apply Lemma \ref{lemma-4-1}: the temperedness assumption for $\{F_n\}$ makes {\bf Requirement 1} hold and
\eqref{condition-2 of lemma} makes {\bf Requirement 2} hold. From the collection of subsets of $F_n$,
  $$\tilde{\mathcal{F}}=\{F_{i,j}a: 1\le i\le M, 1\le j\le N_i \text{ and }a\in A_{i,j}\},$$
we can find by Lemma \ref{lemma-4-1} a subcollection $\mathcal{F}$ which is $10\delta^{1/4}$-disjoint and
\begin{align}\label{condition-F}
  |\cup \mathcal{F}|\ge (\alpha-\delta^{1/4})|F_n|\ge (1-\delta-\delta^{1/4})|F_n|.
\end{align}
In fact, this subcollection $\mathcal{F}$ is a $10\delta^{1/4}$ quasi-tiling of $F_n$ subordinate to $y\in Y$
(see for example \cite{KL} for the detail definition of quasi-tiling).

There may exist overlaps between elements in $\mathcal {F}$. Since $\mathcal{F}$ is $10\delta^{1/4}$-disjoint,
there exits $T'\subset T$ for each $T\in \mathcal {F}$ such that $|T'|/|T|\ge 1-10\delta^{1/4}$ and the collection $\{T': T\in \mathcal {F}\}$ is disjoint.
Denote this new collection by $\mathcal {F}'$.
By \eqref{condition-F},
\begin{align}\label{condition-F1}
  |\cup \mathcal{F}|\le\sum_{T\in\mathcal{F}}|T|\le \frac{1}{1-10\delta^{1/4}}\sum_{T\in\mathcal{F}}|T'|\le \frac{1}{1-10\delta^{1/4}}|F_n|
\end{align}
and
\begin{align}\label{condition-F2}
  |\cup\mathcal{F}'|= \sum_{T\in\mathcal{F}}|T'|\ge (1-10\delta^{1/4})(1-\delta-\delta^{1/4})|F_n|.
\end{align}

{\bf Step 3.} Cover $\pi^{-1}B_{F_n}(y,\eta,\rho)$ through $F_n$-Bowen balls in $(X,G)$.

{\bf Claim.} When $\delta$ is sufficiently small, for any $y\in Y$ and $n>N$, there exists $l(y)>0$ and $v_1(y),v_2(y),\cdots, v_{l(y)}(y)\in X$ such that
$$\bigcup_{i=1}^{l(y)}B_{F_n}(v_i(y),4\varepsilon,d)\supseteq \pi^{-1}(B_{F_n}(y,\eta,\rho))$$
and
$$l(y)\le\exp\big((a+2\tau)|F_n|\big ).$$

\begin{proof}[Proof of the Claim.]
For each $T=F_{i,j}a\in \mathcal{F}$, since $a\in A_{i,j}$, by the construction of $A_{i,j}$, there exists
some point in $\{y_{i,1}, y_{i,2}, \cdots, y_{i,r_i}\}$, denoted by $y_T$, such that $\overline{B}(ay,\eta,\rho)\subseteq W_{y_T}$ and $F_{m(y_T)}=F_{i,j}$.

In the following we will recover the $F_n$-orbits of $(X,G)$ from the $T$-orbits.

Let $E\subset X$ be any finite $\varepsilon$-spanning set under the metric $d$. For any sequence of points $\{z_T\}_{T\in\mathcal {F}}$ with each $z_T\in E_{y_{T}}$ and any sequence of points $\{z_g\}_{g\in F_n\setminus \cup\mathcal {F}'}$ with each $z_g\in E$, let
\begin{align*}
  V(y;\{z_T\}, \{z_g\}):=\{u\in X: d_{T'}(u,a^{-1}z_T)&<2\varepsilon \text{ for all }T=F_{i,j}a\in \mathcal{F}, \\
 &d(gu,z_g)<2\varepsilon \text{ for all }g\in F_n\setminus \cup\mathcal {F}'\}.
\end{align*}
It is not hard to verify that
\begin{align*}
\bigcup_{\{z_T\}, \{z_g\}}V(y;\{z_T\}, \{z_g\})\supseteq\pi^{-1}(B_{F_n}(y,\eta,\rho)),
\end{align*}
i.e. the family $\{V(y;\{z_T\}, \{z_g\}):z_T\in E_{y_{T}}, T\in\mathcal {F}, z_g\in E, g\in F_n\setminus \cup\mathcal {F}'\}$ forms an open cover of $\pi^{-1}(B_{F_n}(y,\eta,\rho)).$ We also note that some of $V(y;\{z_T\}, \{z_g\})$'s may be empty.

We pick any point $v(\{z_T\}, \{z_g\})$ in each non-empty $V(y;\{z_T\}, \{z_g\})$, then
$$B_{F_n}(v(\{z_T\}, \{z_g\}), 4\varepsilon, d)\supseteq V(y;\{z_T\}, \{z_g\}).$$
Enumerate these $v(\{z_T\}, \{z_g\})$'s by $y_1,y_2,\cdots,y_{l(y)}$. We then obtain
\begin{align}\label{4-2}
\bigcup_{i=1}^{l(y)}B_{F_n}(y_i, 4\varepsilon, d)\supseteq\pi^{-1}(B_{F_n}(y,\eta,\rho)).
\end{align}

Now the only thing left is to estimate $l(y)$. Clearly,
\begin{align*}
  l(y)&\le \prod_{T\in\mathcal{F}}|E_{y_T}|\cdot\prod_{g\in F_n\setminus \cup\mathcal {F}'}|E|
       = \prod_{T\in\mathcal{F}}r_{F_{m(y_T)}}(\pi^{-1}(y),\varepsilon,d)\cdot |E|^{|F_n|-|\cup\mathcal {F}'|}\\
      &\le \exp\bigg(\sum_{T\in\mathcal{F}}|F_{m(y_T)}|(a+\tau)+(|F_n|-|\cup\mathcal {F}'|)\log|E|\bigg)\\
      &\le \exp\bigg( \big(\frac{1}{1-10\delta^{1/4}}(a+\tau)+(1-(1-10\delta^{1/4})(1-\delta-\delta^{1/4}))\log|E|\big)|F_n|\bigg)\\
      &\qquad\qquad\qquad\qquad\qquad\qquad\qquad\qquad\qquad\qquad\qquad \text{ (by \eqref{condition-F1} and \eqref{condition-F2})}\\
      &\le \exp\big((a+2\tau)|F_n|\big )\qquad\qquad\text{ (when }\delta \text{ is sufficiently small). }
\end{align*}

\end{proof}

Now we proceed to prove \eqref{eq-4-1}.

For any subset $E$ of $X$, let $H$ be an $(F_n,\eta)$-spanning set of $\pi(E)$ with minimal cardinality
$r_{F_n}(\pi(E),\eta,\rho)$. Then by the above claim, the set $R=\{v_i(y):1\le i\le l(y), y\in H\}$ forms an $(F_n, 4\varepsilon)$-spanning set of $E$, since
\begin{align*}
  \bigcup_{y\in H}\bigcup_{i=1}^{l(y)}B_{F_n}(v_i(y),4\varepsilon,d)&\supseteq \bigcup_{y\in H}\pi^{-1}(B_{F_n}(y,\eta,\rho))\supseteq \pi^{-1}\pi(E)\supseteq E.
\end{align*}
Hence
\begin{align}\label{span-factor}
  r_{F_n}(E,4\varepsilon,d)\le r_{F_n}(\pi(E),\eta,\rho)\cdot\exp\big((a+2\tau)|F_n|\big ).
\end{align}
This deduces that
\begin{align*}
  h^{UC}_{top}( E,\{F_n\})\leq  h^{UC}_{top}(\pi (E),\{F_n\})+a+2\tau.
\end{align*}
Letting $\tau$ tend to $0$, \eqref{eq-4-1} is proved.
\end{proof}

\begin{proof}[Proof of the upper bound of $h^P_{top}( E,\{F_n\})$]\

What we need in fact is the inequality \eqref{span-factor} in the proof of Theorem \ref{thm-ineq-uc}.

From the fact \eqref{sep-span}, we firstly covert inequality \eqref{span-factor} into
\begin{align*}
  s_{F_n}(E,8\varepsilon,d)\le s_{F_n}(\pi(E),\eta,\rho)\cdot\exp\big((a+2\tau)|F_n|\big ).
\end{align*}
This implies that
\begin{align*}
  h^{UC}_{top}(E,8\varepsilon,\{F_n\}) \leq h^{UC}_{top}(\pi(E),\eta,\{F_n\})+a+2\tau.
\end{align*}

By Proposition \ref{packing-uc} (2) again, for any $\delta>0$, there exists a cover $\cup_{i=1}^{\infty}V_i=\pi(E)$ such that
$$h^P_{top}(\pi(E),\eta/3,\{F_n\})+\delta\ge \sup_{i\ge 1} h^{UC}_{top}(V_i,\eta,\{F_n\}).$$
Using the similar argument in the proof of the lower bound,
\begin{align*}
  h^P_{top}(E,8\varepsilon,\{F_n\})&\le\sup_{i\ge 1} h^{P}_{top}(\pi^{-1}(V_i),8\varepsilon,\{F_n\})\le\sup_{i\ge 1} h^{UC}_{top}(\pi^{-1}(V_i),8\varepsilon,\{F_n\})\\
                            &\le\sup_{i\ge 1} h^{UC}_{top}(V_i,\eta,\{F_n\})+a+2\tau\le h^P_{top}(\pi(E),\eta/3,\{F_n\})+\delta+a+2\tau.
\end{align*}
Hence
\begin{align*}
   h^{P}_{top}(E,\{F_n\}) \leq h^{P}_{top}(\pi(E),\{F_n\})+a+2\tau.
\end{align*}
Since $\tau>0$ is arbitrary, we finally obtain
\begin{align*}
   h^{P}_{top}(E,\{F_n\}) \leq h^{P}_{top}(\pi(E),\{F_n\})+\sup_{y\in Y}h_{top}^{UC}(\pi^{-1}(y),\{F_n\}).
\end{align*}
\end{proof}

%%%%%%%%%%%%%%%%%%%%%%%%%%%%%%%%%%%%%%%%%%%%%%%%%%%%%%%%%%%%%%%%%%%%%%%%%%%%%%%%%%%%%%%%%%%%%%%%%%%%%%%%%%%%%%%%%%%%
%%%%%%%%%%%%%%%%%%%%%%%%%%%%%%%%%%%%%%%%%%%%%%%%%%%%%%%%%%%%%%%%%%%%%%%%%%%%%%%%%%%%%%%%%%%%%%%%%%%%%%%%%%%%%%%%%%%%
\section{Amenable packing entropy for certain subsets}

\subsection{The set of generic points}
Recall that for $\mu\in M(X,G)$ and a F{\o}lner sequence $\{F_n\}$ in $G$, the set of generic points for $\mu$
(w.r.t. $\{F_n\}$) is defined by
\begin{align*}
  G_{\mu}=\{x\in X: \lim_{n\rightarrow+\infty}\frac{1}{|F_n|}\sum_{g\in F_n}f(gx)=\int_Xf d\mu, \text{ for any }f\in C(X)\}.
\end{align*}

\begin{remark}\label{remark-5-1}
  If $\mu\in E(X,G)$ and $\{F_n\}$ is a tempered F{\o}lner sequence then $\mu(G_{\mu})=1$. To show this, let $\{f_i\}_{i=1}^{\infty}$ be a countable dense subset of $C(X)$ and denote by
  $$X_i=\{x\in X: \lim_{n\rightarrow+\infty}\frac{1}{|F_n|}\sum_{g\in F_n}f_i(gx)=\int_Xf_i d\mu\}.$$
  By the pointwise ergodic theorem, $\mu(X_i)=1$. Hence $G_{\mu}=\cap_{i=1}^\infty X_i$ has full measure.
\end{remark}
\iffalse
%%%%%%%%%%%%%%%%%%%%%%%%%%%%%%%%%%%%%%%%%%%%

\begin{proof}
  We first assume the F{\o}lner sequence $\{F_n\}$ is tempered. Let $\{f_i\}_{i=1}^{\infty}$ be a countable dense subset of $C(X)$.
  By the pointwise ergodic theorem for amenable groups, for each $f_i$,
  \begin{align*}
    \lim_{n\rightarrow +\infty}\frac{1}{|F_n|}\sum_{g\in F_n}f_i(gx)=\int_Xf_id\mu, \, \mu\text{-a.e.}.
  \end{align*}
  Let $X_i$ be the set of points $x$ such that the above equality holds and set $X_0=\cap_{i=1}^\infty X_i$. Then $\mu(X_0)=1$.
  Let $\varepsilon>0$ be fixed. Then for any $f\in C(X)$, there exists $f_k$ such that $||f-f_k||<\varepsilon$. Now for any $x\in X_0$, we have
  \begin{align*}
    |\frac{1}{|F_n|}\sum_{g\in F_n}f(gx)-\int_Xfd\mu|&\le |\frac{1}{|F_n|}\sum_{g\in F_n}f(gx)-\frac{1}{|F_n|}\sum_{g\in F_n}f_k(gx)|\\
    &\qquad +|\frac{1}{|F_n|}\sum_{g\in F_n}f_k(gx)-\int_Xf_kd\mu|+|\int_Xf_kd\mu-\int_Xfd\mu|\\
    &< ||f-f_k||+\varepsilon+||f-f_k|| \, (\text{when }n \text{ is sufficiently large})\\
    &<3\varepsilon.
  \end{align*}
 Hence $G_{\mu}=X_0$.

 For the case that $\{F_n\}$

\end{proof}
\fi
%%%%%%%%%%%%%%%%%%%%%%%%%%%%%%%%%%%%%%%%%%

The system $(X,G)$ is said to have {\it almost specification} property if there exists a non-decreasing function $g: (0,1)\rightarrow (0,1)$ with $\lim_{r\rightarrow 0}g(r)=0$ (a mistake-density function) and a map $m: (0,1)\rightarrow F(G)\times (0,1)$ such that for any $k\in \N$, any
$\varepsilon_1, \varepsilon_2,\cdots, \varepsilon_k\in (0,1)$, and any $x_1, x_2,\cdots,x_k\in X$, if $F_i$ is $m(\varepsilon_i)$-invariant, $i=1, 2,\cdots, k$
and $\{F_i\}_{i=1}^k$ are pairwise disjoint, then
$$\bigcap_{1\le i\le m}B(g; F_i,x_i,\varepsilon_i)\neq\emptyset,$$
where $B(g; F,x,\varepsilon):=\big\{y\in X: |\{h\in F: d(hx,hy)>\varepsilon\}|\le g(\varepsilon)|F|\big\}$, the Bowen ball allowing a mistake with density $g(\varepsilon)$.

\begin{remark}
It was shown in \cite{Zh} that the weak specification implies the almost specification.
Recall that in \cite{CL} the system $(X,G)$ (the group $G$ need not to be amenable) has {\it weak specification} if for any $\varepsilon>0$ there exists a nonempty finite subset $F$ of $G$
with the following property: for any finite collection $F_1, \cdots, F_m$ of finite subsets $G$ with
$$FF_i\cap F_j=\emptyset \text{ for } 1\le i,j\le m,i\neq j,$$
and for any collection of points $x_1,\cdots,x_m \in X$, there exists a point $y\in X$ such that
$$d(gx_i,gy)\le \varepsilon \text{ for all } g\in F_i, 1\le i\le m,$$
i.e.
$$\bigcap_{1\le i\le m}\overline{B}_{F_i}(x_i,\varepsilon)\neq\emptyset.$$
\end{remark}

In this section, we will prove Theorem \ref{th-generic}, i.e. if $\mu\in M(X,G)$, the F{\o}lner sequence $\{F_n\}$ satisfies the growth condition \eqref{eq-1-1} and
either $\mu$ is ergodic and $\{F_n\}$ is tempered or $(X,G)$ has almost specification, then
\begin{align}\label{5-1}
  h^P_{top}(G_{\mu}, \{F_n\})=h_{\mu}(X,G).
\end{align}

The idea of the proof comes from Pfister and Sullivan \cite{PS} (see also \cite{ZC2,Zh} for amenable group actions).

\subsubsection{Upper bound for $h_{top}^P(G_{\mu},\{F_n\})$}
\

In the following we are going to prove $h_{top}^P(G_{\mu},\{F_n\})\leq h_{\mu}(X,G)$ assuming that the F{\o}lner sequence $\{F_n\}$ satisfies the growth condition \eqref{eq-1-1}.

For $\mu\in M(X,G)$, let $\{K_m\}_{m\in\N}$ be a decreasing sequence of closed convex neighborhoods of $\mu$ in $M(X)$ such that $\bigcap_{m\in\N}K_m=\{\mu\}$. Let
$$A_{n,m}=\{x\in X:\frac{1}{|F_n|}\sum_{g\in F_n}\delta_x\circ g^{-1}\in K_m\}, \text{ for }m,n\in\N,$$
and
$$R_{N,m}=\{x\in X:\text{ for any }n>N, \frac{1}{|F_n|}\sum_{g\in F_n}\delta_x\circ g^{-1}\in K_m\}, \text{ for }m,N\in\N.$$
Then for any $m,N\ge 1$,
$$R_{N,m}=\bigcap_{n>N}A_{n,m} \text{ and } G_{\mu}\subseteq \bigcup_{k> N}R_{k,m}.$$

For $\varepsilon>0$ and $Z\subseteq X$, recall that $s_{F_n}(Z,\varepsilon)$ denotes the maximal cardinality of any $(F_n,\varepsilon)$-separated subset of $Z$.
%Denote by $N(A_{n,m},n,\varepsilon)$ the maximal cardinality of any $(F_n,\varepsilon)$-separated subset of $A_{n,m}$.
Then we have
\begin{align}\label{5-2}
 \limsup_{n\rightarrow +\infty}\frac{1}{|F_n|}\log s_{F_n}(R_{N,m},\varepsilon)
  \le \limsup_{n\rightarrow +\infty}\frac{1}{|F_n|}\log s_{F_n}(A_{n,m},\varepsilon), \text{ for any }m,N\ge 1.
% &=\lim_{\varepsilon\rightarrow 0}\lim_{m\rightarrow +\infty}\limsup_{n\rightarrow +\infty}\frac{1}{|F_n|}\log r_{F_n}(R_{n,m},\varepsilon).
\end{align}
%Hence for any $\eta>0$, there exists $\varepsilon_1>0$ such that for any $0<\varepsilon<\varepsilon_1$,
%\begin{align*}
%  h_{top}^{UC}(R_{N,m},\{F_n\})<\limsup_{n\rightarrow\infty}\frac{1}{|F_n|}\log s_{F_n}(A_{n,m},\varepsilon)+\eta.
%\end{align*}

By the claim in \cite[page 878]{ZC2} (we note that it also works for non-ergodic invariant measures),
\begin{align*}%\label{5-2-1}
\lim_{\varepsilon\rightarrow 0}\lim_{m\rightarrow+\infty}\limsup_{n\rightarrow+\infty}\frac{1}{|F_n|}\log s_{F_n}(A_{n,m},\varepsilon)\le h_{\mu}(X,G).
\end{align*}

Hence for any $\eta>0$, there exists $0<\varepsilon_1$ such that for any $0<\varepsilon<\varepsilon_1$, there
exists $M=M(\varepsilon)\in \N$ such that,
\begin{align*}
  \limsup_{n\rightarrow+\infty}\frac{1}{|F_n|}\log s_{F_n}(A_{n,m},\varepsilon)< h_{\mu}(X,G)+\eta,
\end{align*}
whenever $m\ge M$. Especially,
\begin{align}\label{5-3}
  \limsup_{n\rightarrow+\infty}\frac{1}{|F_n|}\log s_{F_n}(A_{n,M},\varepsilon)< h_{\mu}(X,G)+\eta.
\end{align}

Joint \eqref{5-2} and \eqref{5-3} together, for any $0<\varepsilon<\varepsilon_1$, we have that for any $N\in\N$,
$$\limsup_{n\rightarrow +\infty}\frac{1}{|F_n|}\log s_{F_n}(R_{N,M},\varepsilon)< h_{\mu}(X,G)+\eta.$$
Since for any $N'\in \N$, $G_{\mu}\subseteq \bigcup_{N\ge N'}R_{N,M}$, by Propositions \ref{prop-basic} and \ref{prop-packing-uc},
$$h_{top}^P(G_{\mu},\varepsilon,\{F_n\})\le \sup_{N\ge N'}h_{top}^P(R_{N,M},\varepsilon,\{F_n\})\le \sup_{N\ge N'}\limsup_{n\rightarrow+\infty}\frac{1}{|F_n|}\log s_{F_n}(R_{N,M},\varepsilon),$$
which follows that
$$h_{top}^P(G_{\mu},\varepsilon,\{F_n\})\le h_{\mu}(X,G)+\eta.$$
Letting $\varepsilon\rightarrow 0$ and then $\eta\rightarrow 0$, we obtain that $h_{top}^P(G_{\mu},\{F_n\})\le h_{\mu}(X,G)$.

\subsubsection{Lower bound for $h_{top}^P(G_{\mu},\{F_n\})$}
\

For the case $\mu$ is ergodic and $\{F_n\}$ is tempered, since $\mu(G_{\mu})=1$, Corollary \ref{coro-packing} gives the lower bound.

For the case $\mu\in M(X,G)$ and the system $(X,G)$ has almost specification property, the proof of the lower bound becomes rather complicated because of the
quasi-tiling techniques for amenable groups. But it was shown in \cite{Zh} that
$h_{top}^B(G_{\mu},\{F_n\})= h_{\mu}(X,G)$. Hence by Proposition \ref{prop-packing-Bowen},
we obtain $h_{top}^P(G_{\mu},\{F_n\})\ge h_{\mu}(X,G)$.

\subsection{$G$-symbolic dynamical system}\

Let $A$ be a finite set with cardinality $|A|\ge 2$ and $A^G=\{(x_g)_{g\in G}: x_g\in A\}$ be the $G$-symbolic space over $A$. Consider the left action of $G$ on $A^G$:
$$g'(x_g)_{g\in G}=(x_{gg'})_{g\in G}, \text{ for all }g'\in G\text{ and }(x_g)_{g\in G}\in A^G.$$
$(A^G,G)$ forms a $G$-symbolic dynamical system or a $G$-acting full shift (over $A$). For any non-empty closed $G$-invariant subset $X$ of $A^G$, the subsystem $(X,G)$ is called a subshift. For $x=(x_g)_{g\in G}$ and a finite subset $F\subset G$, denote by $x|_F=(x_g)_{g\in F}\in A^F$ the restriction of $x$ to $F$ and
denote by $[x|_F]=\{\omega\in A^G: \omega_g=x_g \text{ for all }g\in F\}$ (which is called a {\it cylinder}).

Fix any tempered F{\o}lner sequence $\{F_n\}$ of $G$ with $F_0=\{e_G\}\subsetneq F_1\subsetneq F_2\subsetneq\ldots$
and $\bigcup_{n}F_n=G$. Note that $\{F_n\}$ satisfies the growth condition \eqref{eq-1-1} automatically. We can then define a metric $d$ on $A^G$ associated to $\{F_n\}$ by the following:
\begin{align}\label{metric}
  d(x,y)=\begin{cases}
    1,\qquad\qquad \text{ if }x \text{ and }y \text{ are not equal on } F_0;\\
    e^{-|F_n|}, \qquad n=\max\{k: x|_{F_k}=y|_{F_k}\}.
  \end{cases}
\end{align}

To discuss the regularity for subsets of $(A^G,G)$, we need to consider the relation between Bowen entropy (packing entropy) and the corresponding Hausdorff dimension (packing dimension). Before that, we recall the definitions of Hausdorff dimension and packing dimension (cf. \cite{M}).

\begin{definition}
  Let $(X,d)$ be a compact metric space. Let $0\le s<\infty$. For $E\subset X$ and $\varepsilon>0$, put
  $$P_{\varepsilon}^s(E)=\sup\sum_i({\rm diam} B_i)^s$$
  where the supremum is taken over all disjoint families of closed balls $\{\overline B_i\}$ such that ${\rm diam} B_i \le \varepsilon$ and the centres of the $B_i$'s are in $E$.

  Then set $P^s(E)=\lim\limits_{\varepsilon\rightarrow 0}P_{\varepsilon}^s(E)$ (since $P_{\varepsilon}^s(E)$ is non-decreasing on $\varepsilon$) and
  define $$\mathcal P^s(E)=\inf \{\sum_{i=1}^{\infty}P^s(E_i): E=\bigcup_{i=1}^{\infty}E_i\}.$$
  The {\it packing dimension} of $E$ is then defined by
  $${\rm dim}_P(E)=\inf\{s: \mathcal P^s(E)=0\}=\sup\{s: \mathcal P^s(E)=\infty\}.$$

  Let $\mathcal H^s(E)=\lim\limits_{\varepsilon\rightarrow 0}\mathcal H_{\varepsilon}^s(E)$, where
  $$\mathcal H_{\varepsilon}^s(E)=\inf\{\sum_{i=1}^\infty({\rm diam} E_i)^s: E\subset \bigcup_{i=1}^\infty E_i, {\rm diam} E_i\le \varepsilon\}.$$
  The {\it Hausdorff dimension} of $E$ is then defined by
  $${\rm dim}_H(E)=\inf\{s: \mathcal H^s(E)=0\}=\sup\{s: \mathcal H^s(E)=\infty\}.$$

\end{definition}

Comparing with the definitions of dimensional entropies, we have the following proposition.
\begin{proposition}\label{prop-5-4}
Let the F{\o}lner sequence $\{F_n\}$ satisfy the following two conditions:
\begin{enumerate}
  \item $F_mF_n\subset F_{m+n}$ for each $m,n\in\N$;
  \item $\lim\limits_{n\rightarrow\infty} \frac{F_{n+1}}{{F_n}}=1$.
\end{enumerate}
Then for any $E\subset A^G$,
\begin{align*}
 h_{top}^B(E,\{F_n\})={\rm dim}_HE\text{ and } h_{top}^P(E,\{F_n\})={\rm dim}_PE,
\end{align*}
where the dimensions ${\rm dim}_H$ and ${\rm dim}_P$ are both under the metric $d$ defined in \eqref{metric}.
\end{proposition}
\begin{proof}
  See Appendix A.1.
\end{proof}
\begin{remark}\begin{enumerate}
  \item Due to \cite{HSt}, the sequence of finite subsets $\{F_n\}$ of $G$ satisfying condition (1) in Proposition \ref{prop-5-4} is called
  a {\it regular system}. If $G$ is a finitely generated group and let $F_n$ be the collection of elements in $G$ with word length (w.r.t. a finite
  symmetric generating subset) no more than $n$,
  then $\{F_n\}$ satisfies condition (1).
  \item There are examples of amenable groups which admit F{\o}lner sequences satisfying conditions in Proposition \ref{prop-5-4}.
  An abelian example of the group $G$ is $\Z^d$, with $F_n=[-n,n]^d$. An non-abelian example of the group $G$ is the
   dihedral group, with $F_n$ which is chosen to be the collection of elements with word length no more than $n$ (see Example 6.4.11 of \cite{CC}).
\end{enumerate}
\end{remark}
By Proposition \ref{prop-5-4}, we have
\begin{proposition}
  Let the F{\o}lner sequence $\{F_n\}$ satisfy the conditions in Proposition \ref{prop-5-4}. Then any subset $E\subset A^G$ is regular in the sense of dimensional entropy if and only if $E$ is dimension-regular (under the F{\o}lner sequence $\{F_n\}$ and metric $d$).
\end{proposition}

By Remark \ref{remark-2.6}, any non-empty closed $G$-invariant subset $X$ of $A^G$ is regular in the sense of dimensional entropy. Hence we have
\begin{corollary}
Let the F{\o}lner sequence $\{F_n\}$ satisfy the conditions in Proposition \ref{prop-5-4}. Then any non-empty closed $G$-invariant subset $X$ of $A^G$ is both regular in the sense of dimensional entropy and dimension-regular (under the F{\o}lner sequence $\{F_n\}$ and metric $d$).
\end{corollary}

Let $0\le \alpha < \beta\le 1$ and $A=\{0,1\}$. Let $H\subset G$ such that
\begin{align*}
  \liminf_{n\rightarrow+\infty}\frac{|H\cap F_n|}{|F_n|}=\alpha \text{ and }\limsup_{n\rightarrow+\infty}\frac{|H\cap F_n|}{|F_n|}=\beta,
\end{align*}
i.e. $H$ is a subset of $G$ with lower density $\alpha$ and upper density $\beta$ w.r.t. $\{F_n\}$. Now we define $X_{\alpha,\beta}\subset \{0,1\}^G$ by
\begin{align*}
  X_{\alpha,\beta}=\{(x_g)_{g\in G}: x_g=0 \text{ if }g\notin H\}.
\end{align*}
Assume in addition that $\{F_n\}$ satisfies the conditions in Proposition \ref{prop-5-4}, then we have
\begin{proposition}\label{prop-5-8}
  $h_{top}^B (X_{\alpha,\beta},\{F_n\})=\alpha\log 2$ and $h_{top}^P(X_{\alpha,\beta},\{F_n\})=\beta\log 2$. Hence $X_{\alpha,\beta}$ is not regular.
\end{proposition}
\begin{proof}
  See Appendix A.2.
\end{proof}
\iffalse
\begin{remark}
  Let $\{F_{t_n}\}$ be a subsequence of $\{F_n\}$ such that $\lim_{n\rightarrow \infty}\frac{|H\cap F_{t_n}|}{|F_{t_n}|}=\alpha$. It is not hard to see that
  $h_{top}^P(X_{\alpha,\beta},\{F_{t_n}\})=\alpha$ while $h_{top}^P(X_{\alpha,\beta},\{F_{n}\})=\beta$. This shows that amenable packing entropy does depend on the choice of F{\o}lner sequences. Similar argument holds for amenable Bowen entropy.
\end{remark}
\fi
\subsection{Fibers of $\{T,T^{-1}\}$ transformation}\

Random walk in random scenery (RWRS) is a class of stationary random processes which are well-studied both in probability theory and ergodic theory. RWRS provides measure-theoretic models with amazingly rich behavior (see \cite{HS,A}). Among the class of RWRS, $\{T,T^{-1}\}$ transformation, although seems simple,
is possibly the best known in the history of ergodic theory since it is a natural example of a $K$-automorphism that is not Bernoulli (\cite{K}). In spite of its measure-theoretic aspect, we will consider the topological model of the $\{T,T^{-1}\}$ transformation and investigate subsets of the topological system.

\begin{definition}[Topological $\{T,T^{-1}\}$ transformation]
Let $A=\{1,-1\}$ and as convention we denote the shift map on $A^{\Z}$ by $T$ which is defined by
\begin{align*}
  (T(x))_i=x_{i+1}, \text{ for any } x=(x_i)_{i\in \Z}\in A^{\Z}.
\end{align*}
The $\{T,T^{-1}\}$ transformation, denoted by $S$, on $A^{\Z}\times A^{\Z}$ is defined by
\begin{align*}
  S(x,y)=\begin{cases}
    &\big(T(x), T(y)\big), \quad\quad\text{ if }y_0=1;\\
    &\big(T^{-1}(x), T(y)\big), \quad\text{ if }y_0=-1.
  \end{cases}
\end{align*}
\end{definition}
Then $S^n(x,y)=\big(T^{\omega(y,n)}(x),T^ny\big)$ for $n\in\Z$, where
\begin{align*}
  \omega(y,n):=\begin{cases}
    &\sum_{j=0}^{n-1}y_j,\quad\quad\quad\text{ if }n>0;\\
    &0,\qquad\qquad\quad\quad \text{ if }n=0;\\
    &-\sum_{j=n}^{-1}y_j, \quad\quad\,\, \text{if }n<0.
  \end{cases}
\end{align*}

Clearly for the system $(A^{\Z}\times A^{\Z},S)$, the acting group $G$ here is the integer group $\Z$. The F{\o}lner sequence $\{F_n\}$ is chosen naturally to be $F_n=\{0,1,\cdots,n-1\}:=[0,n-1]$.
Define a metric $\rho$ on $A^{\Z}\times A^{\Z}$ by
\begin{align*}
  \rho\big((x,y),(x',y')\big)=\max \{d(x,x'),d(y,y')\},
\end{align*}
where $d$ is the metric on $A^{\Z}$ defined by
\begin{align*}
  d(x,y)=2^{-n}, \text{ where }n=\min\{|i|: x_i\neq y_i\}.
\end{align*}
We note that the metric $d$ here is different from \eqref{metric}.

Let $\pi: A^{\Z}\times A^{\Z}\rightarrow A^{\Z}$ be the projection to the second coordinate. Then it induces a factor map between $(A^{\Z}\times A^{\Z},S)$
and $(A^{\Z},T)$. For any $y\in A^{\Z}$, let $E_{y}:=A^{\Z}\times \{y\}$ denote the fiber of $y$ under the factor map $\pi$. We denote by
$$h_{top}^P(E_{y},S)=h_{top}^P(E_{y},\{F_n\})\text{ and } h_{top}^B(E_{y},S)=h_{top}^B(E_{y},\{F_n\}),$$
the packing and Bowen entropies of $E_y$ for the $\Z$-system
$(A^{\Z}\times A^{\Z},S)$ respectively.

For $n>0$, denote by
\begin{align*}
  M(y,n)=\max_{0\le i\le n}\omega(y,i) \text{ and }m(y,n)=\min_{0\le i\le n}\omega(y,i).
\end{align*}

\begin{proposition}\label{prop-5-11}
\begin{enumerate}
\item $h_{top}^P(E_{y},S)=\limsup\limits_{n\rightarrow+\infty}\frac{M(y,n)-m(y,n)}{n}h_{top}(A^{\Z},T);$
  \item $h_{top}^B(E_{y},S)=\liminf\limits_{n\rightarrow+\infty}\frac{M(y,n)-m(y,n)}{n}h_{top}(A^{\Z},T)$.
\end{enumerate}
Here $h_{top}(A^{\Z},T)(=\log 2)$ is the topological entropy of the symbolic dynamical system $(A^{\Z},T)$. Hence $E_y$ is regular if and only if
$\lim\limits_{n\rightarrow+\infty}\frac{M(y,n)-m(y,n)}{n}$ exists.
\end{proposition}
\begin{proof}
  See Appendix A.3.
\end{proof}

{\bf Acknowledgements}
The authors are grateful to Prof. Wen Huang for his helpful discussions. This research is supported by NNSF of China (Grant No. 11790274, 12171233, 11701275, 12171175 and 11801193) and Tianyuan Mathematical Center in Southwest China.

%%%%%%%%%%%%%%%%%%%%%%%%%%%%%%%%%%%%%%%%%%%%%%%%%%%%%%%%%%%%%%%%%%%%%%%%%%%%%%%%%%%%%%%%%%%%%%%%%%%%%%%%%%%%%%%%%%%%%%%%%%%%%%%%%%%%%%
%%%%%%%%%%%%%%%%%%%%%%%%%%%%%%%%%%%%%%%%%%%%%%%%%%%%%%%%%%%%%%%%%%%%%%%%%%%%%%%%%%%%%%%%%%%%%%%%%%%%%%%%%%%%%%%%%%%%%%%%%%%%%%%%%%%%%%
\appendix
\renewcommand{\appendixname}{Appendix~\Alph{section}}
\section{Proofs of Proposition \ref{prop-5-4}, \ref{prop-5-8} and \ref{prop-5-11}}

 In this appendix we will give the detailed proofs of Propositions \ref{prop-5-4}, \ref{prop-5-8} and \ref{prop-5-11}.

\subsection{Proof of Proposition \ref{prop-5-4}}\

Let $E$ be a subset of the compact metric space $(A^G,d)$ as defined in Section 5.2. Recall that in Proposition \ref{prop-5-4} the F{\o}lner sequence $\{F_n\}$ satisfies the following two conditions:
\begin{enumerate}
  \item $F_mF_n\subseteq F_{m+n}$ for each $m,n\in\N$;
  \item $\lim\limits_{n\rightarrow\infty} \frac{F_{n+1}}{{F_n}}=1$.
\end{enumerate}

We divide the proof of Proposition \ref{prop-5-4} into two parts.

{\bf Part 1.  Proof of $h_{top}^B(E,\{F_n\})={\rm dim}_HE$.}
\begin{proof}
  Let $s>h_{top}^B(E,\{F_n\})$. From the definition of Bowen entropy, we have
  $$\mathcal{M}(E,\varepsilon,s,\{F_n\})=0, \text{ for any }\varepsilon>0.$$
  Hence for any $N>0$, $\mathcal{M}(E,N,1,s,\{F_n\})=0$. For any $\delta>0$, there exists a countable familiy $\{B_{F_{n_i}}(x^i,1)\}$ with
$x^i\in E, n_i\ge N$ and $\bigcup\limits_i B_{F_{n_i}}(x^i,1)\supset E$ such that
$$\sum_ie^{-s|F_{n_i}|}<\delta.$$
Notice that
\begin{align}\label{A-1}
  B_{F_{n_i}}(x^i,1)&=\{y\in A^{G}:\rho(gx^i,gy)<1,\forall g\in F_{n_i}\}\nonumber\\
  &= \{y\in A^{G}:(gx^i)_{e_G}=(gy)_{e_G},\forall g\in F_{n_i}\}\nonumber\\
  &=[x^i|_{F_{n_i}}]
\end{align}
and ${\rm diam}[x^i|_{F_{n_i}}]=e^{-|F_{n_i}|}\le e^{-|F_N|}$. Since the family of cylinders $\{[x^i|_{F_{n_i}}]\}$ covers $E$,
we have
$$\mathcal H_{e^{-|F_N|}}^s(E)\le \sum_ie^{-s|F_{n_i}|}<\delta.$$
Therefore $\mathcal H_{e^{-|F_N|}}^s(E)=0$ and then $\mathcal H^s(E)=0$. This means that ${\rm dim}_HE\le s$. Since $s>h_{top}^B(E,\{F_n\})$ is arbitrary,
$h_{top}^B(E,\{F_n\})\ge {\rm dim}_HE$.

Now we will show $h_{top}^B(E,\{F_n\})\le {\rm dim}_HE$.

If $h_{top}^B(E,\{F_n\})=0$, then there is nothing to prove. Assume $h_{top}^B(E,\{F_n\})>0$ and let $0<s<h_{top}^B(E,\{F_n\})$,
then there exists $0<\varepsilon<1$ such that
\begin{align}\label{A-2}
  \mathcal M(E,\varepsilon,s,\{F_n\})>1.
\end{align}
Assume $\varepsilon\in (e^{-|F_k|},e^{-|F_{k-1}|}]$ for some $k\in \N$ and let $\eta>0$ be fixed. By condition (2), there exists
$N'>k$ such that
$$\frac{|F_n|}{|F_{n-k}|}<1+\eta, \text{ whenever }n>N'.$$
By \eqref{A-2}, there exists $N>N'$ such that $\mathcal M(E,N,\varepsilon,s,\{F_n\})>1$.

Let $\{E_i\}_{i=1}^\infty$ be any countable family that covers $E$ and ${\rm diam}E_i<e^{-|F_{N+k}|}$ for each $i$.
By the definition of the metric $d$, ${\rm diam}E_i=e^{-|F_{n_i}|}$ for some $n_i>N+k$. Choose any point $x^i\in E_i$,
then $E_i\subset [x^i|_{F_{n_i}}]$.

Noticing that
\begin{align}\label{A-3}
  B_{F_{n_i-k}}(x^i,\varepsilon)&=\{y\in A^G: \rho(gx^i,gy)<\varepsilon, \text{ for all }g\in F_{n_i-k}\}\nonumber\\
  &=\{y\in A^G: \rho(gx^i,gy)\le e^{|F_k|}, \text{ for all }g\in F_{n_i-k}\} \nonumber \\
  &=\{y\in A^G: (gx^i)|_{F_k}=(gy)|_{F_k}, \text{ for all }g\in F_{n_i-k}\} \nonumber \\
  &=\{y\in A^G: x^i|_{F_kF_{n_i-k}}=y|_{F_kF_{n_i-k}}\}=[x^i|_{F_kF_{n_i-k}}] \nonumber \\
  &\supseteq [x^i|_{F_{n_i}}]\; (\text{ since }F_kF_{n_i-k}\subseteq F_{n_i}),
\end{align}
the family $\{B_{F_{n_i-k}}(x^i,\varepsilon)\}_{i=1}^\infty$ also covers $E$ (with each $n_i-k>N$).
So $$\sum_{i=1}^\infty e^{-s|F_{n_i-k}|}\ge \mathcal M(E,N,\varepsilon,s,\{F_n\})>1.$$
And hence
\begin{align*}
  \sum_{i=1}^\infty e^{-\frac{s}{1+\eta}|F_{n_i}|}=\sum_{i=1}^\infty e^{-\frac{s}{1+\eta}\frac{|F_{n_i}|}{|F_{n_i-k}|}|F_{n_i-k}|}>\sum_{i=1}^\infty e^{-s|F_{n_i-k}|}>1,
\end{align*}
which implies that $\mathcal H^{\frac{s}{1+\eta}}(E)\ge \mathcal H_{e^{-|F_{N+k}|}}^{\frac{s}{1+\eta}}(E)\ge 1$.
Thus ${\rm dim}_HE\ge \frac{s}{1+\eta}$.

Letting $\eta\rightarrow 0$, we have ${\rm dim}_HE\ge s$. Since $0<s<h_{top}^B(E,\{F_n\})$ is chosen arbitrarily, we obtain that
$h_{top}^B(E,\{F_n\})\le {\rm dim}_HE$.

\end{proof}

{\bf Part 2. $h_{top}^P(E,\{F_n\})={\rm dim}_PE$.}
\begin{proof}
We show $h_{top}^P(E,\{F_n\})\ge {\rm dim}_PE$ firstly.

Let $s>h_{top}^P(E,\{F_n\})$. Then for any $0<\varepsilon<1$ it holds that $h_{top}^P(E,\varepsilon,\{F_n\})<s$ and
hence $\mathcal P(E,\varepsilon, s,\{F_n\})=0$. So for any $\delta>0$, there exists a countable covering $\{E_i\}_{i=1}^\infty$ of $E$
such that $$\sum_{i=1}^\infty P(E_i,\varepsilon, s,\{F_n\})<\delta.$$
For each $i$, we can find $N_i\in\N$ sufficiently large such that
$$P(E_i,N_i,\varepsilon, s,\{F_n\})<P(E_i,\varepsilon, s,\{F_n\})+\frac{\delta}{2^i}.$$
Let $\{B_{i,j}\}_{j=1}^\infty$ be a family of disjoint closed balls in $A^G$ (with centers $x^{i,j}\in E_i$ and ${\rm diam} B_{i,j}\le e^{-|F_{N_i}|}$)
such that
$$P_{e^{-|F_{N_i}|}}^s(E_i)\le \sum_{j=1}^\infty ({\rm diam }B_{i,j})^s+\frac{\delta}{2^i}.$$
From the definition of the metric $d$, ${\rm diam }B_{i,j}=e^{-|F_{n_{i,j}}|}$ for some $n_{i,j}\ge N_i$.
Noticing that
\begin{align*}
  B_{i,j}&=B(x^{i,j}, e^{-|F_{n_{i,j}}|})=[x^{i,j}|_{F_{n_{i,j}}}]\\
  &\supseteq \overline B_{F_{n_{i,j}}}(x^{i,j},\varepsilon)\,\,(\text{here we have assumed }\varepsilon<1),
\end{align*}
$\{\overline B_{F_{n_{i,j}}}(x^{i,j},\varepsilon)\}_{j=1}^\infty$ is also a pairwise disjoint family.
Hence
\begin{align*}
  \mathcal P^s(E)&\le \sum_{i=1}^\infty P^s(E_i)\le \sum_{i=1}^\infty P_{e^{-|F_{N_i}|}}^s(E_i)
  \le \sum_{i=1}^\infty \bigg(\sum_{j=1}^\infty e^{-s|F_{n_{i,j}}|}+\frac{\delta}{2^i}\bigg)\\
  &\le \sum_{i=1}^\infty P(E_i, N_i, \varepsilon, s, \{F_n\}) +\delta
  < \sum_{i=1}^\infty \bigg (P(E_i, \varepsilon, s, \{F_n\}) +\frac{\delta}{2^i} \bigg )+\delta\\
  &<3\delta.
\end{align*}
Thus $\mathcal P^s(E)=0$ and ${\rm dim}_PE\le s$. Since $s>h_{top}^P(E,\{F_n\})$ is arbitrary, we obtain ${\rm dim}_PE\le h_{top}^P(E,\{F_n\})$.

Next we will show $h_{top}^P(E,\{F_n\})\le {\rm dim}_PE$.

Assume $h_{top}^P(E,\{F_n\})>0$, otherwise there is nothing to prove. Let $0<s<h_{top}^P(E,\{F_n\})$. Then there exists $0<\varepsilon <1$ such that
$h_{top}^P(E,\varepsilon,\{F_n\})>s$. Assume $\varepsilon\in (e^{-|F_k|},e^{-|F_{k-1}|}]$ for some $k\in \N$ and let $\eta>0$ be fixed. Similar to Part 1, by condition (2), there exists
$N>k$ such that
$$\frac{|F_n|}{|F_{n-k}|}<1+\eta, \text{ whenever }n>N.$$
Let $\{E_i\}_{i=1}^\infty$ be any countable family that covers $E$, then
$$\sum_{i=1}^\infty P(E_i,\varepsilon,s,\{F_n\})\ge \mathcal P(E,\varepsilon,s,\{F_n\})=\infty.$$
For each $i$, let $N_i>N$ be sufficiently large such that
$$P_{e^{-|F_{N_i}|}}^{\frac{s}{1+\eta}}(E_i)<P^{\frac{s}{1+\eta}}(E_i)+\frac{1}{2^i}.$$
Let $\{\overline B_{F_{n_{i,j}}}(x^{i,j},\varepsilon)\}_{j=1}^\infty$ be a disjoint family with $x^{i,j}\in E_i$ and $n_{i,j}\ge N_i$ for each $j$
such that
$$P(E_i,N_i,\varepsilon,s, \{F_n\})<\sum_{j=1}^\infty e^{-s|F_{n_{i,j}}|}+\frac{1}{2^i}.$$
With similar discussion as \eqref{A-3}, $\overline B_{F_{n_{i,j}}}(x^{i,j},\varepsilon)\supseteq [x^{i,j}|_{F_{n_{i,j}+k}}]$ and hence
$\{[x^{i,j}|_{F_{n_{i,j}+k}}]\}_{j=1}^\infty$ is a disjoint family of closed balls with
$${\rm diam} [x^{i,j}|_{F_{n_{i,j}+k}}]=e^{-|F_{n_{i,j}+k}|}\le e^{-|F_{N_i+k}|}.$$
Therefore
$$P_{e^{-|F_{N_i+k}|}}^{\frac{s}{1+\eta}}(E_i)\ge \sum_{j=1}^\infty e^{-\frac{s}{1+\eta}|F_{n_{i,j}+k}|}.$$
Thus we have
\begin{align*}
  \sum_{i=1}^\infty P^{\frac{s}{1+\eta}}(E_i)&>\sum_{i=1}^\infty P_{e^{-|F_{N_i+k}|}}^{\frac{s}{1+\eta}}(E_i)-1
  \ge \sum_{i=1}^\infty \sum_{j=1}^\infty e^{-\frac{s}{1+\eta}|F_{n_{i,j}+k}|} -1\\
  &= \sum_{i=1}^\infty \sum_{j=1}^\infty e^{-\frac{s}{1+\eta}\frac{|F_{n_{i,j}+k}|}{|F_{n_{i,j}}|}|F_{n_{i,j}}|} -1
  \ge \sum_{i=1}^\infty \sum_{j=1}^\infty e^{-s|F_{n_{i,j}}|} -1 \\
  &> \sum_{i=1}^\infty \bigg( P(E_i,N_i,\varepsilon,s, \{F_n\})-\frac{1}{2^i}\bigg )-1\\
  &\ge \sum_{i=1}^\infty P(E_i,\varepsilon,s, \{F_n\})-2=\infty,
\end{align*}
which implies that $\mathcal P^{\frac{s}{1+\eta}}(E)=\infty$ and then ${\rm dim}_P E\ge \frac{s}{1+\eta}$.

Letting $\eta\rightarrow 0$, we have ${\rm dim}_P E\ge s$. Since $0<s<h_{top}^P(E,\{F_n\})$ is chosen arbitrarily, we obtain that
$h_{top}^P(E,\{F_n\})\le {\rm dim}_P E$.
\end{proof}

\subsection{Proof of Proposition \ref{prop-5-8}}\

Recall that for Proposition \ref{prop-5-8} we let $0\le \alpha < \beta\le 1$, $A=\{0,1\}$ and let $H\subset G$ such that
\begin{align*}
  \liminf_{n\rightarrow+\infty}\frac{|H\cap F_n|}{|F_n|}=\alpha \text{ and }\limsup_{n\rightarrow+\infty}\frac{|H\cap F_n|}{|F_n|}=\beta.
\end{align*}
$X_{\alpha,\beta}\subset \{0,1\}^G$ is defined by
\begin{align*}
  X_{\alpha,\beta}=\{(x_g)_{g\in G}: x_g=0 \text{ if }g\notin H\}.
\end{align*}

Let $\mu\in M(X_{\alpha,\beta})$ such that $\displaystyle \mu([x|_{F_n}])=\frac{1}{2^{|H\cap F_n|}}$ for every $x\in  X_{\alpha,\beta}$ and $n\in\N$.
Noticing that $B_{F_n}(x,1)=[x|_{F_n}]$ (see \eqref{A-1}), we have
\begin{align*}
  \underline h_{\mu}^{loc}(X_{\alpha,\beta},\{F_n\})&=\int_{X_{\alpha,\beta}}\lim_{\varepsilon\rightarrow 0}\liminf_{n\rightarrow +\infty}-\frac{1}{|F_n|}\log \mu(B_{F_n}(x,\varepsilon))d\mu\\
  &\ge\int_{X_{\alpha,\beta}}\liminf_{n\rightarrow +\infty}-\frac{1}{|F_n|}\log \mu([x|_{F_n}])d\mu\\
  &=\liminf_{n\rightarrow +\infty}-\frac{1}{|F_n|}\log \frac{1}{2^{|H\cap F_n|}}=\alpha \log 2,
\end{align*}
and similarly,
\begin{align*}
  \overline h_{\mu}^{loc}(X_{\alpha,\beta},\{F_n\})\ge\beta\log 2.
\end{align*}
Applying Theorem 3.1 of \cite{ZC} (the variational principle for amenable Bowen entropy) and Theorem \ref{th-viriational-pack} (the variational principle for amenable packing entropy) respectively, we have
$$h_{top}^B (X_{\alpha,\beta},\{F_n\})\ge \alpha\log 2 \text{ and }h_{top}^P(X_{\alpha,\beta},\{F_n\})\ge\beta\log 2.$$

To prove $h_{top}^B (X_{\alpha,\beta},\{F_n\})\le \alpha\log 2$, by Proposition \ref{prop-5-4}, we only need to prove that
${\rm dim}_H(X_{\alpha,\beta})\le \alpha\log 2$.

Let
\begin{align}\label{En}
  E_n=\{(x_g)_{g\in G}: x_g=0 \text{ if }g\notin H\cap F_n\},
\end{align}
which is a subset of $X_{\alpha,\beta}$ with cardinality $\# E_n=2^{|H\cap F_n|}$. Note that $\bigcup_{x\in E_n} [x|_{F_n}]\supseteq X_{\alpha,\beta}$ and
${\rm diam}[x|_{F_n}]=e^{-|F_n|}$ for each $x\in E_n$. Let $\delta>0$ and $N\in \N$ be fixed. When $n>N$,
$$\mathcal H_{e^{-|F_N|}} ^{(\alpha+\delta)\log 2}(X_{\alpha,\beta}) \le \sum_{x\in E_n}e^{-(\alpha+\delta)\log 2 |F_n|}=e^{\log 2|F_n|(\frac{|H\cap F_n|}{|F_n|}-\alpha-\delta)}.$$
Since $\liminf_{n\rightarrow+\infty}\frac{|H\cap F_n|}{|F_n|}=\alpha$, we have
$$\mathcal H_{e^{-|F_N|}} ^{(\alpha+\delta)\log 2}(X_{\alpha,\beta})=0\text{ and }\mathcal H^{(\alpha+\delta)\log 2}(X_{\alpha,\beta})=\lim_{N\rightarrow \infty}\mathcal H_{e^{-|F_N|}} ^{(\alpha+\delta)\log 2}(X_{\alpha,\beta})=0.$$
Thus
$${\rm dim}_H(X_{\alpha,\beta})\le (\alpha+\delta)\log 2,$$
which deduces that
$${\rm dim}_H(X_{\alpha,\beta})\le \alpha\log 2.$$

Finally we will prove that $h_{top}^P (X_{\alpha,\beta},\{F_n\})\le \beta\log 2$.

Let $\varepsilon>0$ such that $\varepsilon\in (e^{-|F_k|},e^{-|F_{k-1}|}]$ for some $k\in\N$.
Noticing that $[x|_{F_{n+k}}]\subseteq B_{F_n}(x,\varepsilon)$ (see \eqref{A-3}, here condition (1) of Proposition \ref{prop-5-4} is used),
one can check that the set $E_{n+k}$, defined by \eqref{En}, is an $(F_n,\varepsilon)$-spanning set of $X_{\alpha,\beta}$.
Hence $r_{F_n}(X_{\alpha,\beta},\varepsilon)\le 2^{|H\cap F_{n+k}|}$. Then
$$\limsup_{n\rightarrow +\infty}\frac{1}{|F_n|}\log r_{F_n}(X_{\alpha,\beta}, \varepsilon)
\le \limsup_{n\rightarrow +\infty}\frac{|H\cap F_{n+k}|}{|F_n|}\log 2=\beta\log 2,$$
where for the last equality we use condition (2) of Proposition \ref{prop-5-4}.
Thus
$$h_{top}^P (X_{\alpha,\beta},\{F_n\})\le h_{top}^{UC}(X_{\alpha,\beta},\{F_n\})\le \beta\log 2.$$

\subsection{Proof of Proposition \ref{prop-5-11}}\

\begin{proof}
 (1). For any $\varepsilon>0$ and $x\in A^{\Z}$, in the system $(A^{\Z}\times A^{\Z},S)$, we have
  \begin{align}\label{5-5}
    B_{n}\big((x,y),\varepsilon,\rho\big)\cap E_y&=\big\{(x',y):\rho\big(S^i(x',y),S^i(x,y)\big)<\varepsilon,\text{ for }0\le i\le n-1\big\}\nonumber\\
    &=\big\{(x',y): d\big(T^{\omega(y,i)}(x'),T^{\omega(y,i)}x\big)<\varepsilon,\text{ for }0\le i\le n-1\big\}\nonumber\\
    &=B_{[m(y,n-1),M(y,n-1)]}\big(x,\varepsilon,d\big)\times\{y\}.
  \end{align}
  Let $\mu$ be the $\{\frac{1}{2},\frac{1}{2}\}$ Bernoulli measure on $(A^{\Z},T)$ and $\delta_y$ be the Dirac probability measure at the point $y$.
  Then
  \begin{align}\label{5-6}
   \overline h_{\mu\times \delta_y}^{loc}(E_y,S)&=\int_{E_y}\lim_{\varepsilon\rightarrow 0}\limsup_{n\rightarrow +\infty}-\frac{1}{n}\log \big(\mu\times \delta_y\big)\bigg(B_{n}\big((x,y),\varepsilon,\rho\big)\bigg)d\big(\mu\times \delta_y\big)\nonumber\\
   &=\int_{A^{\Z}}\lim_{\varepsilon\rightarrow 0}\limsup_{n\rightarrow +\infty}-\frac{1}{n}\log\mu\bigg(B_{[m(y,n-1),M(y,n-1)]}\big(x,\varepsilon,d\big)\bigg)d\mu\nonumber\\
   &=\limsup_{n\rightarrow+\infty}\frac{M(y,n)-m(y,n)}{n}h_{top}(A^{\Z},T),
  \end{align}
  where for the last inequality we use the simple facts that $M(y,n-1)\le M(y,n)\le M(y,n-1)+1$ and $m(y,n-1)-1\le m(y,n)\le m(y,n-1)$.
  Hence by Theorem \ref{th-viriational-pack}, we have
  \begin{align*}
     h_{top}^P(E_{y},S)&\ge \overline h_{\mu\times \delta_y}^{loc}(E_y,S)=\limsup_{n\rightarrow\infty}\frac{M(y,n)-m(y,n)}{n}h_{top}(A^{\Z},T).
  \end{align*}

  Let $E\times \{y\}$ be any $(n,\varepsilon)$-separated set for $E_y$. Note that from \eqref{5-5} $E$ must be an $([m(y,n-1),M(y,n-1)],\varepsilon)$-separated set for $A^{\Z}$.
We have
\begin{align*}
  h_{top}^{UC}(E_y,S)&\le \limsup_{n\rightarrow+\infty}\frac{M(y,n)-m(y,n)}{n}h_{top}^{UC}(A^{\Z},T)\\
  &=\limsup_{n\rightarrow+\infty}\frac{M(y,n)-m(y,n)}{n}h_{top}(A^{\Z},T).
\end{align*}
By Proposition \ref{prop-packing-uc},
  \begin{align*}
     h_{top}^P(E_{y},S)&\le h_{top}^{UC}(E_y,S)\le\limsup_{n\rightarrow\infty}\frac{M(y,n)-m(y,n)}{n}h_{top}(A^{\Z},T).
  \end{align*}
This finishes the proof of (1).

(2). Similar to \eqref{5-6}, we have
\begin{align*}
   \underline h_{\mu\times \delta_y}^{loc}(E_y,S)&=\int_{E_y}\lim_{\varepsilon\rightarrow 0}\liminf_{n\rightarrow +\infty}-\frac{1}{n}\log \big(\mu\times \delta_y\big)\bigg(B_{n}\big((x,y),\varepsilon,\rho\big)\bigg)d\big(\mu\times \delta_y\big)\\
   &=\int_{A^{\Z}}\lim_{\varepsilon\rightarrow 0}\liminf_{n\rightarrow +\infty}-\frac{1}{n}\log\mu\bigg(B_{[m(y,n-1),M(y,n-1)]}\big(x,\varepsilon,d\big)\bigg)d\mu\\
   &=\liminf_{n\rightarrow+\infty}\frac{M(y,n)-m(y,n)}{n}h_{top}(A^{\Z},T).
  \end{align*}
Hence by Theorem 3.1 of \cite{ZC}, the variational principle for amenable Bowen entropy, we have
\begin{align*}
     h_{top}^B(E_{y},S)\ge \underline h_{\mu\times \delta_y}^{loc}(E_y,S)=\liminf_{n\rightarrow\infty}\frac{M(y,n)-m(y,n)}{n}h_{top}(A^{\Z},T).
  \end{align*}

Now we are to prove the upper bound for $h_{top}^B(E_{y})$, i.e.
$$h_{top}^B(E_{y},S)\le \liminf_{n\rightarrow\infty}\frac{M(y,n)-m(y,n)}{n}h_{top}(A^{\Z},T).$$

Let $\delta>0$ be fixed. For any $\varepsilon>0$, there exists $k\in\N$ such that for any $x,z\in A^{\Z}$, whenever $x_i=z_i$ for every $|i|\le k$,
it holds that $d(x,y)<\varepsilon$. Hence for any interval $[m.n]$ of integers, we have
$$B_{[m,n]}(x,\varepsilon,d)\supseteq [x|_{[-k+m,k+n]}],$$
where $[x|_{[-k+m,k+n]}]:=\{z\in A^{\Z}: z_i=x_i, \text{ for every }-k+m\le i\le k+n\}$ is the cylinder in $A^{\Z}$.

Let
$$E_n=\{x\in A^{\Z}: x_i=1 \text{ if }i\notin [-k+m(y,n-1),k+M(y,n-1)]\}.$$
Consider the family
$$\{B_{n}\big((x,y),\varepsilon,\rho\big)\cap E_y\}_{x\in E_n}.$$
Apparently, it covers $E_y$, since $B_{n}\big((x,y),\varepsilon,\rho\big)\cap E_y=B_{[m(y,n-1),M(y,n-1)]}\big(x,\varepsilon,d\big)\times\{y\}$ (by \eqref{5-5}).
Hence for any $N\in \N$,  we have for any $n\ge N$,
\begin{align*}
  \mathcal M(E_y,N,\varepsilon,s,\{F_n\})&\le (\# E_n) e^{-sn}\\
  &=2^{M(y,n-1)-m(y,n-1)+2k+1}e^{-sn}\\
  &=e^{n (\frac{M(y,n-1)-m(y,n-1)+2k+1}{n}\log 2-s)}.
\end{align*}
Note that there exist infinitely many $n\in\N$ such that
$$\frac{M(y,n-1)-m(y,n-1)+2k+1}{n}<\liminf\limits_{n\rightarrow+\infty}\frac{M(y,n)-m(y,n)}{n}+\delta.$$
Then for any $s>(\liminf\limits_{n\rightarrow+\infty}\frac{M(y,n)-m(y,n)}{n}+\delta)\log 2$, we can deduce that
$$\mathcal M(E_y,N,\varepsilon,s,\{F_n\})=0.$$
Letting $N\rightarrow +\infty$, $\varepsilon\rightarrow 0$ and then $\delta\rightarrow 0$,
from the definition of the Bowen entropy, we can conclude that
$$h_{top}^B(E_{y},S)\le \liminf_{n\rightarrow\infty}\frac{M(y,n)-m(y,n)}{n}h_{top}(A^{\Z},T).$$
This finishes the proof of (2) of Proposition \ref{prop-5-11}.
\end{proof}

\end{document}